\documentclass[reqno]{amsart}

\usepackage[utf8]{inputenc}
\usepackage[english]{babel}
\usepackage[margin=1.6in]{geometry}
\usepackage[T1]{fontenc}
\usepackage{amsmath,amsfonts,amssymb,amsthm}
\usepackage{thmtools}
\usepackage{enumerate}
\usepackage{enumitem}
\usepackage{palatino}
\usepackage{euler}
\usepackage{calrsfs}
\usepackage{tikz}
\usepackage{graphicx}
\usepackage{relsize}
\usepackage{multicol}
\usepackage{relsize}
\usepackage{accents} % Para agregar widetildes no tan grandes
\usepackage{bm}
\usepackage{dirtytalk}
\usepackage{lipsum}
\usepackage{tcolorbox}
\usepackage{tikz-cd,calc}
\usetikzlibrary{patterns,decorations.pathreplacing,babel}
\usetikzlibrary{shapes,intersections}
\usepackage{hyperref}
\hypersetup{
    colorlinks=true,
    linkcolor=purple,
    filecolor=magenta,      
    citecolor=cyan,
    }

\newcommand\C{\mathbb{C}}

\newcommand\Z{\mathbb{Z}}

\newcommand\A{\mathbb{A}}

\newcommand\m{\mathfrak{m}}

\renewcommand\P{\mathbb{P}}

\newcommand{\OO}{\mathcal{O}}

\newcommand{\Hilb}{\operatorname{Hilb}}

\newcommand{\X}{\mathcal{X}}
\newcommand{\Y}{\mathcal{Y}}
\newcommand{\B}{\mathcal{B}}
\newcommand{\F}{\mathscr{F}}
\newcommand{\dF}{\mathcal{F}}
\newcommand{\G}{\mathscr{G}}
\newcommand{\dG}{\mathcal{G}}

\newcommand\E{\mathcal{E}}

\renewcommand\v{\raise0.9ex\hbox{$\scriptscriptstyle\vee$}}

\DeclareMathOperator{\Hom}{Hom}

\DeclareMathOperator{\Spec}{Spec}
\DeclareMathOperator{\Ex}{Ex}
\DeclareMathOperator{\Der}{Der}
\DeclareMathOperator{\Ext}{Ext}
\DeclareMathOperator{\Aut}{Aut}

\DeclareMathOperator{\Sing}{Sing}
\DeclareMathOperator{\codim}{codim}
\DeclareMathOperator{\rank}{rank}
\DeclareMathOperator{\im}{Im}

\DeclareMathOperator{\Fol}{Fol}
\DeclareMathOperator{\Sym}{Sym}
\DeclareMathOperator{\Drap}{Drap}
\DeclareMathOperator{\Quot}{Quot}

\DeclareMathOperator{\Unf}{Unf}

\DeclareMathOperator{\Bl}{Bl}

\DeclareMathOperator{\Tang}{Tang}

\DeclareFontFamily{OT1}{pzc}{}
\DeclareFontShape{OT1}{pzc}{m}{it}{<-> s * [1.100] pzcmi7t}{}
\DeclareMathAlphabet{\mathpzc}{OT1}{pzc}{m}{it}

\DeclareMathOperator{\sHom}{\mathpzc{Hom}}

\DeclareMathOperator{\fDrap}{\mathpzc{Drap}}
\DeclareMathOperator{\fQuot}{\mathpzc{Quot}}

\usepackage{hyperref}

\newtheorem{theorem}{Theorem}[section]
\newtheorem{proposition}[theorem]{Proposition}
\newtheorem{lemma}[theorem]{Lemma}
\newtheorem{corollary}[theorem]{Corollary}

\theoremstyle{definition}

\newtheorem{definition}[theorem]{Definition}
 \newtheorem{remark}[theorem]{Remark}

\newtheorem{example}[theorem]{Example}

\newtheoremstyle{named}{}{}{\itshape}{}{\bfseries}{.}{.5em}{\thmnote{#3}}

\theoremstyle{named}

\title{Deformations of pullback foliations}
\author{Pablo Perrella}
\date{}

\address{Departamento de Matem\'atica 
, Facultad de Ciencias Exactas y Naturales, Universidad de Buenos Aires 
and IMAS (UBA-CONICET) 
, Buenos Aires, Argentina.}
\email{pperrella@dm.uba.ar} 


\begin{document}

\maketitle

\begin{abstract}
We study the stability of pullback foliations under morphisms and rational maps via Grothendieck's Drapeaux scheme. In the local setting, a foliated version of Schlessinger's Theorem on rigidity of conical singularities was achieved. We apply these techniques to provide a criterion for the stability of algebraic leaves of a foliation. 
\end{abstract}

\section*{Introduction}

One of the central topics of study in deformation theory is the stability property of a family of geometric objects. Stable families are those that contain all sufficiently small perturbations of their members. In this article, we address the problem of stability for pullback foliations  on normal algebraic varieties defined over $\C$. 
Informally speaking, given a dominant rational map $\pi: X\dashrightarrow Y$ and a foliation $\G$ on $Y$, the pullback $\F=\pi^\ast\G$ is the foliation whose leaves are the pre-images of the leaves of $\G$. As a particularly interesting special case, if $\G$ is a foliation by points, the fibers of the rational map are the leaves of the pullback foliation.

%\textcolor{red}{When the map $\pi$ is a morphism, it is easier to determine when a given deformation (see Definitions~\ref{def:family} and~\ref{def:def})  of $\F$ over a variety $\X$ is stable.}

Previous research has addressed the case where the map $\pi$ is a morphism. Girbau, Haeflieger and Sundararaman proved that certain foliations given by fibrations are stable, within the framework of transversely holomorphic deformations of smooth foliations \cite[Thm 4.3.1]{girbau1983deformations}. In the singular setting, Gomez-Mont and Lins Neto proved the $C^\infty$-stability of a class of foliations given by the fibers of a Morse function $\pi: X\rightarrow C$ to  a complex curve $C$~\cite[Thm 1.5]{gomez1991structural}. In both cases, the topological condition $H^1(F,\C)= 0$ is imposed on the fibers of the morphism under consideration. In this direction, we were able to establish the following result, which we prove in Section~\ref{ssec:proofT01}.

\begin{theorem}{\label{stability_morphisms}}
Let $\pi: X\rightarrow Y$ be a proper, surjective morphism with connected fibers between normal varieties whose generic fiber is smooth and has no non-zero global holomorphic $1$-forms. Suppose that
\[
0\rightarrow T_{\dF/S}\rightarrow T_{\X/S}\rightarrow N_{\dF/S}\rightarrow 0
\]
is a deformation of a pullback foliation $\F = \pi^\ast\G$ over a smooth base $S$. If the family $\X$ is projective over $S$ and $\Pi: \X\rightarrow \Y$ is a deformation of $\pi$ over $S$, then there exists an étale neighborhood $U\rightarrow S$ and a deformation $\dG/U$ of $\G$ such that $\dF/U = \Pi^\ast(\dG/U)$.
\end{theorem}

The stability problem for pullback foliations under rational maps grows in complexity and was studied by several authors. In \cite{gomez1991structural} the stability of foliations with general rational first integrals $\pi: X\dashrightarrow \P^1$ was settled. Suitable pullbacks of foliations by curves on the projective plane under maps $\pi:\P^m\dashrightarrow \P^2$ are also stable \cite{pullbackCLE}. Infinitesimal techniques were applied in \cite{cukiermanRacionales} (in contrast to the topological arguments of the previous results) not only to give a second proof of the stability of rational foliations, but also to deduce that the corresponding irreducible component of the moduli space of foliations is generically reduced. More recently, the stability of general pullbacks of split foliations on weighted projective spaces under rational maps $\pi: \P^m\dashrightarrow\P_w$ was proved applying an infinitesimal approach \cite{pullbackConFede}. Theorem~\ref{stability_morphisms} provides a third method, scheme theoretical in nature, to prove a partial generalization of these results:

\begin{theorem}{\label{stability_rational_maps}}
Let $X\subseteq \P^N$ be a smooth projective subvariety, $\pi: X\dashrightarrow\P^n$ a dominant rational map with $\OO_X(1) = \pi^\ast\OO(1)$, and $\G$ a split foliation of codimension one on $\P^n$ such that
%\begin{multicols}{2}
\begin{enumerate}[label=(\alph*)]
\item $\dim X\geqslant n+2$,
%\item $\Gamma\big(\Omega^1_X\big) = 0$,
\item $X$ is smooth and arithmetically Cohen-Macaulay,
\item $(\pi,\G)$ is a generic pair (Definition \ref{generic_pair}),
\item the non-Kupka locus of $\G$ has codimension $\geqslant 3$.
\end{enumerate}
%\end{multicols}
Consider a deformation of the pullback foliation $\F=\pi^\ast\G$ given by the exact sequence
\[
0\rightarrow I_{\dF/S}\rightarrow \Omega_{X\times S/S}^{1} \rightarrow \Omega^{1}_{\dF/S}\rightarrow 0.
\]
If $S$ is a smooth curve, then there exists an analytic neighborhood $U\subseteq S$, a rational map $\Pi~:~X\times~U~\dashrightarrow~\P^n$ extending $\pi$ and a deformation $\dG/U$ of $\G$ such that $\dF/U=\Pi^\ast(\dG/U)$.
\end{theorem}

One of the main ingredients of the proof of Theorem~\ref{stability_rational_maps}, that we present in Section~\ref{ssec:proofT02}, is the stability of conical foliation singularities. This problem was addressed in \cite{camacho1982topology} for cones of foliations by curves, and applied latter in the previously mentioned main result of \cite{pullbackCLE}. Deformations of conical singularities of varieties were studied first in \cite{schlessinger1973rigid}. It was proved that, given a smooth projective subvariety $X\subseteq \P^n$ such that
\[
H^1\big(X,T_X(\ell)\big) = 0 \quad\forall \ell\neq 0,
\]
deformations of the germ at the origin of the affine cone $C(X)$ are obtained as cones of deformations of the inclusion $X\subseteq \P^n$. It is worth to notice that these cohomology groups parametrize \textit{thickenings} of $X$, respectively bent according to the line bundles $\OO_X(\ell)$ \cite[p.15-16]{sernesi2007deformations}. Turning back to the foliated scenario, the\textit{ cone of a foliation} $\F$ on $\P^n$ is the pullback $C(\F):= \pi^\ast\F$ under the quotient map $\pi: \C^{n+1}\dashrightarrow \P^n$. In analogy with the concept of thickening, the notion of \textit{unfolding of a foliation} was introduced by Suwa and explored throughout several articles \cite{suwa1985determinacy,suwa1992unfoldings,molinuevo2016unfoldings}. Infinitesimal unfoldings are parametrized by a graded module $\Unf_\F$ (see Proposition \ref{inf_unfoldings} and the paragraph below). In this way we could obtain a foliated version of Schlessinger's Theorem for conical singularities:

\begin{theorem}{\label{stabcones}}
Suppose that $\F$ is a foliation of codimension one and degree $k-2$ on $\P^n$ without polynomial integrating factors such that
\[
\Unf_{\F}(\ell) = 0\quad\; \forall\;\ell \neq k.
\]
If $\dF/S$ is the universal family of foliations of codimension one and degree $k-2$ on $\P^n$, then the cone of this family $\big(C(\dF)/S,0\big)$ is a versal deformation of $(C(\F),0)$.
\end{theorem}
The proof of Theorem~\ref{stabcones} can be found in Section~\ref{ssec:proofT03}.
A collection of examples of foliations $\F$ verifying the hypothesis in this theorem are those of \textit{split type}, i.e, whose tangent sheaf $T_\F$ is a direct sum of line bundles. More concretely, a codimension one split foliation on $\P^n$ whose non-Kupka locus has codimension at least three has~$\Unf_\F = 0$ \cite[p.21]{massri2015kupka}. As a consequence, \cite{CukiermanSplit} provides a list of families of \textit{rigid} foliation singularities.

On the other hand, we were also able to apply the techniques behind Theorem \ref{stability_morphisms} to give a criterion for the stability of foliations having algebraic leaves in the next theorem, which is proved in Section~\ref{ssec:proofT04}.

\begin{theorem}{\label{stab_leaves}}
Let $\F$ be a codimension one foliation on a smooth projective variety $X$, and $Z$ a smooth unobstructed subvariety of $X$. Suppose $Z$ is an algebraic leaf of $\F$ such that $$H^0\big(Z,\Omega^1_Z\otimes N_{Z/X}\big) = H^0\big(Z,\Omega^1_Z\otimes N_{Z/X}^2\big) = 0.$$
Given a deformation $\dF/S$ of $\F$ of the form
\[
0\rightarrow T_{\dF/S}\rightarrow T_{X\times S/S}\rightarrow N_{\dF/S}\rightarrow 0,
\]
there exists an étale neighborhood $U\rightarrow S$ and a deformation $\mathcal{Z}\subseteq X\times U$ of the subvariety $Z\subseteq X$ such that $\mathcal{Z}_u$ is an algebraic leaf of $\dF|_u$ for all $u\in U$.
\end{theorem}

We briefly explain the main ideas of the proofs of Theorems~\ref{stability_morphisms} and~\ref{stab_leaves}. Both of them rely on the \textit{Drapeaux schemes} parametrizing flags of subsheaves of a given coherent sheaf, a construction due to Grothendieck. Given a projective scheme $\X/S$ and $\E$ a coherent sheaf on $\X$ that is flat over $S$, there is a scheme \smash{$\Drap^\ell_\E$} that represents the functor $\operatorname{Sch_{/S}}\rightarrow \operatorname{Set}$ which assigns to an arrow $T\rightarrow S$ the set of flags of $\OO_{X\times_S T}$-modules
\[
0 = F_0 \subseteq F_1 \subseteq \cdots \subseteq F_{\ell-1}\subseteq F_{\ell} = \E\otimes_{\OO_S}\OO_T
\]
such that the successive quotients $F_i/F_{i-1}$ are flat over $T$. 

Concerning the ideas behind Theorem~\ref{stability_morphisms}, pullback foliations of the form $\F = \pi^\ast\G$ are exactly those for which the inclusion $T_{X/Y}\subseteq T_\F$ holds \cite[Lem 2.22]{cerveau2006algebraic}. If that is the case, we obtain a filtration
\[
T_{X/Y}\subseteq T_{\F} \subseteq T_{X}.
\]
corresponding to a point \smash{$p\in \Drap_{T_\X}^3$}. On the other hand, the forgetful map defined by $(F_1\subseteq F_2\subseteq T_{\X})\mapsto (F_2\subseteq T_{\X})$ induces a morphism \smash{$\Drap_{T_\X}^3\rightarrow \Quot_{T_\X}$}. A deformation $T_{\dF/S}\subseteq T_\X$ gives a map \smash{$S\rightarrow \Quot_{T_\X}$}, and the desired stability properties of pullback foliations would imply the existence (at least étale locally) of a lifting \smash{$S\rightarrow \Drap_{T_\X}^3$} passing through the point $p$. We apply infinitesimal deformation theory to prove that \smash{$\Drap_{T_\X}^3\rightarrow \Quot_{T_\X}$} is smooth at $p$ under the mentioned hypothesis on the fibers of $\pi:X\rightarrow Y$, and therefore we can conclude the existence of such lifting.
This method can be also used to study stability properties of algebraic leaves. The key observation is that a subvariety $Z\subseteq X$ is an algebraic leaf of $\F$ whenever the inclusion $T_{\F}\subseteq T_{X}(-\log Z)$ holds. Hence similar considerations apply to the flag
\[
T_{\F}\subseteq T_{X}(-\log Z)\subseteq T_X,
\]
and the map \smash{$\Drap_{T_\X}^3\rightarrow \Quot_{T_\X}$} witch forgets the second element of a length three flag.\\

\textbf{Acknowledgements.} This work is part of the author's PhD thesis at Universidad de Buenos Aires under the supervision of Fernando Cukierman, fully supported by CONICET. We would like to thank Fernando for his generous guidance, which has been invaluable not only academically but also as a mentor in a broader sense. We would like to express our sincere gratitude to Alicia Dickenstein and Sebastián Velazquez for kindly dedicating their time to help improve this work. We also wish to thank Henrique Bursztyn, Hossein Movasati and Claudia Reynoso Alcántara for their useful comments and suggestions.

%\tableofcontents

% Capítulos del trabajo
\section{Preliminaries} 

In this introductory section we recall some notions on singular foliations and deformation theory. We summarize some basic results on obstruction theory of local rings.

\begin{definition}
A \textit{(singular) foliation} $\F$ of codimension $q$ on a normal variety $X$ consists of an exact sequence of coherent bundles
\[
0\rightarrow T_{\F}\rightarrow T_X\rightarrow N_\F\rightarrow 0
\]
such that $N_\F$ is torsion-free of rank $q$ and $[T_\F,T_\F]\subseteq T_\F$. We will call $T_{\F}$ and $N_{\F}$ respectively the \textit{tangent} and \textit{normal sheaf} of the foliation $\F$.
\end{definition}

Let \smash{$\Omega^{[p]}_X := (\Omega^{p}_X)^{\vee\vee}$} be the sheaf of reflexive holomorphic $p$-forms. A codimension $q$ foliation $\F$ can be alternatively described as a short exact sequence of coherent sheaves
\[
0\rightarrow I_{\F}\rightarrow \Omega_{X}^{[1]} \rightarrow \Omega^{[1]}_{\F}\rightarrow 0
\]
with $\Omega^{[1]}_\F$ torsion-free and, in a neighborhood of a general point, there is a local basis $\omega_1,\dots,\omega_q\in I_\F$ satisfying the \textit{Frobenius integrability conditions}
\[
d\omega_i \wedge \omega_1\wedge\cdots\wedge\omega_q = 0\quad(1\leqslant i \leqslant q).
\]

More explicitly, if we dualize the exact sequence that determines $\F$ we obtain a left exact sequence
\smash{$0\rightarrow N_\F^\vee \rightarrow \Omega_{X}^{[1]} \rightarrow T_{\F}^\vee$} and therefore the previous sheaves turn out to be $I_\F := N_\F^\vee$ and \smash{$\Omega^{[1]}_{\F} := \operatorname{Im}\big(\Omega^{[1]}_{X}\rightarrow T_{\F}^\vee\big)$}. These two ways of representing $\F$ are equivalent since if we dualize the exact sequence of differential forms we obtain the original sequence of vector fields of the definition of foliation.\\

Given a foliation $\F$ of codimension q, the quotient $T_X\rightarrow N_\F$ induces a morphism \smash{$(\wedge^q T_X)^{\vee\vee} \rightarrow (\wedge^q N_\F)^{\vee\vee} = \det N_\F$}, and therefore a twisted $q$-form \smash{$\omega\in \Gamma\big(\Omega^{[q]}_X(\det N_{\F})\big)$}.

\begin{definition} \label{sing_fol}
The \textit{singular locus} of $\F$ is the scheme $\Sing (\F) = \Sing (\omega)$, that is, the zero locus of the twisted $q$-form $\omega$.
\end{definition} 

\begin{definition} Let $\pi:X\rightarrow Y$ be a morphism between normal varieties and $\G$ a foliation on $Y$. The \textit{pullback} $\F = \pi^\ast \G$ is defined as that foliation on $X$ whose tangent sheaf $T_\F$ is the kernel of the composition
\[
T_X \rightarrow \pi^{[\ast]} T_Y \rightarrow \pi^{[\ast]} N_\G,
\]
where \smash{$\pi^{[\ast]}(-) = (\pi^{\ast}(-))^{\vee\vee}$} is the reflexive pullback. More generally, if $\pi: X\dashrightarrow Y$ is a rational map with domain $U\subseteq X$, we can define the pullback $\F = \pi^\ast\G$ as the unique foliation whose restriction to $U$ coincides with $(\pi|_U)^\ast \G$. 
\end{definition} 

It is clear from the definition that a necessary condition for $\F$ to be the pullback by a morphism $\pi: X\rightarrow Y$ of a foliation on $Y$ is that there is an inclusion between tangent spaces $T_{X/Y}\subseteq T_{\F}$. The following result says that this condition is also sufficient.

\begin{lemma}[{{\cite[Lem 2.22]{cerveau2006algebraic}}}]{\label{jorge}}
Let $\pi: X\rightarrow Y$ be a surjective morphism between normal varieties with connected fibers, and $\F$ be a foliation on $X$. If the tangent sheaf $T_\F$ contains $T_{X/Y}$, then there exists a foliation $\G$ on $Y$ such that $\F = \pi^\ast\G$.
\end{lemma}

\subsection{Deformations of foliations} We follow closely the definitions and ideas presented in \cite{Quallbrunn_2015}. In order to introduce the notion of families of foliations over a non-reduced base, we need replacement of the torsion-free hypothesis on the normal sheaf $N_\F$ that behaves well over non integral schemes:

\begin{definition}
We will say that a sheaf of $\OO_X$-modules $N$ on a scheme $X$ is \textit{torsionless} if the natural map $N\rightarrow N^{\vee\vee}$ is a monomorphism.
\end{definition}

\begin{definition}\label{def:family}
Let $\X/S$ be a locally trivial family of normal varieties. A \textit{family of integrable (not necessarily saturated) distributions} over $\X/S$ consists of a short exact sequence of coherent sheaves
\[
0\rightarrow T_{\dF/S}\rightarrow T_{\X/S}\rightarrow N_{\dF/S}\rightarrow 0
\]
such that $N_{\dF/S}$ is torsionless and $T_{\dF/S}$ is closed under the relative Lie bracket of $\X/S$.
\end{definition}

When the base $S$ is integral, the torsionless condition on $N_{\dF/S}$ is equivalent to being torsion-free. However, both notions are different in general. The reason why we will stick with the former is that these kinds of families admit an equivalent description in terms of differential forms. Let us denote by \smash{$\Omega^{[p]}_{\X/S} =(\Omega^{p}_{\X/S})^{\vee\vee}$} the sheaf of \textit{relative reflexive p-forms}.

\begin{definition}
A \textit{family of involutive (not necessarily saturated) Pfaff systems} on $\X/S$ consists of a short exact sequence of coherent sheaves
\[
0\rightarrow I_{\dF/S}\rightarrow \Omega_{\X/S}^{[1]} \rightarrow \Omega^{[1]}_{\dF/S}\rightarrow 0
\]
with \smash{$\Omega^{[1]}_{\dF/S}$} torsionless and, in a neighborhood of a general point, there is a local basis $\omega_1,\dots,\omega_q$ of $I_{\dF/S}$ satisfying  
$$d\omega_i\wedge\omega_1\wedge\dots\wedge\omega_q = 0\quad (1\leqslant i\leqslant q),$$ 
where we denoted by $d$ the relative de Rham differential of $\X/S$.
\end{definition}

Following the same procedure as in the preceding sections, to every family of integrable distributions there corresponds via duality a family of involutive Pfaff systems and vice-versa \cite[Prop. 2.30]{warner1983foundations}. Here we are implicitly using that a subsheaf of a torsionless sheaf is also torsionless and that the dual of any $\OO_X$-module bundle is torsionless as well. Since normal and conormal sheaves are torsionless, the proof of \cite[Lem. 4.2]{Quallbrunn_2015} guarantees that this is a 1-1 correspondence. \\

Although these represent two aspects of the same underlying mathematical structure, a distinction in terminology for these types of families is justified for the following reasons:\\

\begin{itemize}
\item We would first need a sensible notion of flat families. As expected, we will say that a family of integrable distributions (resp. involutive Pfaff systems) is \textit{flat} if the sheaf \smash{$N_{\F/S}$} (resp. \smash{$\Omega^{[1]}_{\F/S}$}) is flat over $S$. Nevertheless, these two conditions of flatness are not equivalent \cite[Section 2]{Quallbrunn_2015}.\\

\item Both types of families form a collection of foliations not necessarily saturated. More precisely, given a closed point $s\in S$ and a flat family of integrable distributions we obtain via restriction to the fiber $\X_s$ a short exact sequence
\[
0\rightarrow \left. (T_{\dF/S})\right\vert_{s}\rightarrow T_{\X_s}\rightarrow \left. (N_{\dF/S})\right\vert_{s}\rightarrow 0.
\]
where the injectivity of the first morphism is guaranteed by the flatness of $N_{\dF/S}$ on the basis $S$ (note that $T_{\X_s} = T_\X|_s$ since $\X/S$ is a locally trivial family). However, the sheaf \smash{$\left.(N_{\dF/S})\right\vert_{s}$} is not necessarily torsion-free, and hence we need to saturate in order to get a foliation. Analogous problems arise for flat families of Pfaff systems as well. \\
\end{itemize}

We can avoid the flatness problem if we assume that $S$ is a smooth curve. In this case $\X$ is integral, both sheaves \smash{$N_{\F/S}$} and \smash{$\Omega_{\dF/S}^{[1]}$} are torsion-free, and hence they are flat over $S$ \cite[Exercise 24.4.B]{vakil2017rising}. This agrees with the example in \cite[Section 2]{Quallbrunn_2015} because it has a surface as a base $S$. Let us illustrate the second of the phenomena.

\begin{example}[{{\cite[Example 54]{tesisSeba}}}]{\label{seba}} %Example 54
Consider the space \smash{$\P^4 = \P(\Sym^4(\C^2))$} of binary forms of degree four together with the action of $\operatorname{SL}_2(\C)$ given by change of coordinates. It determines a foliation $\F$ on $\P^4$ whose leaves are the closures of the codimension one orbits. The classical $j$-invariant is a rational first integral of $\F$, and therefore is defined by the $1$-form $\omega_0 = 3g_0df_0-2f_0dg_0$, with $f_0$ and $g_0$ homogeneous polynomials of degrees 2 and 3 respectively \cite[p.7]{calvo2006note}. On the other hand \smash{$T_\F = \mathfrak{sl}_2(\C)\otimes_{\C}\OO_{\P^4}\subseteq T_{\P^4}$}, because the action is free in codimension two.

Let $f_s$ and $g_s$ be two generic deformations of the above polynomials, parameterized by the affine line $S = \A^1$. We then obtain a flat family of involutive Pfaff systems
\[
0\rightarrow \OO_{\P^4\times S}(-5)\rightarrow \Omega^1_{\P^{4}\times S/S}\rightarrow \Omega^{1}_{\dF/S}\rightarrow 0,
\]
where the morphism on the left corresponds to the 1-form $\omega_s = 3g_sdf_s-2f_sdg_s$. Since $\omega_0$ does not vanish in codimension one, the sheaf \smash{$(\Omega^{1}_{\dF/S})\big\vert_{0}$} is torsion-free. Nevertheless \smash{$\left.(N_{\dF/S})\right\vert_{0}$} has torsion. If this was not the case this restriction should coincide with the normal sheaf $N_{\F}$, and \smash{$\left.(T_{\dF/S})\right\vert_{0}$} with $T_\F\simeq \OO_{\P^4}^3$. But this last bundle is rigid and consequently $\omega_s$ would have trivial tangent sheaf for values of $s\in \A^1$ close to zero. This is a contradiction because split foliations have equidimensional singular locus, while a generic rational foliation has isolated singularities \cite{cukierman2006singularities}.
\end{example}

This last step suggests why the dualization process fails: the singular loci of the previous foliations do not assemble into a flat family. This phenomenon was studied in detail in \cite{Quallbrunn_2015}. In general we define the \textit{singular locus} of a family of integrable distributions/involutive Pfaff systems of codimension $q$ as the zero locus $\Sing(\dF)$ of the twisted $q$-form \smash{$\omega\in \Gamma\big(\Omega^{[q]}_{\X/S}(\det N_{\dF/S})\big)$} analogously obtain as in Definition \ref{sing_fol}.

\begin{definition} \label{def:def}
We will say that a flat family of integrable distributions (resp. flat family of involutive Pfaff systems) $\dF/S$ is a \textit{deformation} of a foliation $\F$ on $X$ if there exists a closed point $s \in S$ and an isomorphism $X\simeq \X_s$ such that the restriction of the exact sequence defining $\dF/S$ to the fiber over $s$ is identified via this isomorphism with the exact sequence of vector fields (resp. differential forms) that defines $\F$.
\end{definition}

\begin{lemma}{\label{dualizar}}
Suppose that $\X/S$ is a smooth family over a smooth curve $S$ and the family of involutive Pfaff systems
\[
0\rightarrow I_{\dF/S}\rightarrow \Omega_{\X/S}^{[1]} \rightarrow \Omega^{[1]}_{\dF/S}\rightarrow 0
\]
is a deformation of a codimension one foliation $\F$. If the singular locus $\Sing(\dF)$ is flat over $S$, then the dual family of integrable distributions
\[
0\rightarrow T_{\dF/S}\rightarrow T_{\X/S}\rightarrow N_{\dF/S}\rightarrow 0
\]
is also a deformation of $\F$.
\end{lemma}

\begin{proof}
Our goal is to demonstrate that \smash{$\left.(N_{\dF/S})\right\vert_{s}$} is torsion-free. This follows immediately from the description \smash{$N_{\dF/S} = I_{\Sing(\dF)} \otimes I_{\dF/S}^\vee$}, as the flatness assumption guarantees that \smash{$\left.\big(I_{\Sing(\dF)}\big)\right\vert_{s} = I_{\Sing(\F)}$}.
\end{proof}

The applicability of this result relies on the existence of \textit{stable singularities} of foliations, which guarantee that the flatness condition on the scheme $\Sing(\dF)/S$ is fulfilled.

\begin{definition}
Given $\F$ a codimension one foliation on $X$, we say that a regular point $p\in X$ is a
\begin{enumerate}
    \item \textit{Morse singularity} if there are analytic coordinates $z_1,\dots,z_n$ around $p$ such that $\F$ is represented locally by the 1-form $\omega = d(z_1^2+\dots+z_n^2)$.
    \item \textit{Kupka singularity} if there is a $1$-form germ $\omega\in \Omega^1_{X,p}$ representing $\F$ such that $\omega_p = 0$ but $d\omega_p \neq 0$.
\end{enumerate}
\end{definition}

\begin{proposition}[{{\cite[p.24]{Quallbrunn_2015}}}]{\label{Kupka&Morse}}
Suppose that
\[
0\rightarrow I_{\dF/S}\rightarrow \Omega_{\X/S}^{[1]} \rightarrow \Omega^{[1]}_{\dF/S}\rightarrow 0
\]
is a deformation of a codimension one foliation $\F$. If $p$ is either a Morse or Kupka singularity of $\F$, then the singular locus $\Sing(\dF)$ is flat over $S$ at the point $p$.
\end{proposition}

Later in Section 5 we will study the class of \textit{conic foliation singularities}, and we will prove that they are also stable as stated in the Theorem \ref{stabcones} of the Introduction. One of the motivations for their study is the structure of the singular locus of certain pullback foliations by a generic rational map $\pi:X\dashrightarrow \P^n$. Under good circumstances such foliations have only Morse, Kupka or conical singularities. This will allow us to use the Lemma \ref{dualizar} to analyze the stability of these foliations. 

Suppose that $\F$ is a codimension one foliation determined by a twisted $1$-form $\omega$ and $\beta: \operatorname{Bl}_p X\rightarrow X$ is the blow-up at a smooth point of $X$. We define the \textit{multiplicity} of $\F$ at $p$ to be the vanishing order $m_p(\F)$ of $\beta^\ast\omega$ along the exceptional divisor $E\subseteq Bl_p X$. Following the definitions we get:

\begin{lemma}{\label{def_blowup}}
Consider a deformation
\[
0\rightarrow I_{\dF/S}\rightarrow \Omega_{X\times S/S}^{[1]} \rightarrow \Omega^{[1]}_{\dF/S}\rightarrow 0
\]
of a codimension one foliation $\F$ with $S$ a smooth curve. Let $\beta: \Bl_pX\rightarrow X$ be the blow-up about a smooth point $p\in X$ such that the multiplicity $m_p(\dF_s)$ is independent of $s\in S$. Then $\beta^\ast (\dF /S)$ is a deformation of the blow-up $\beta^\ast\F$.
\end{lemma}

\subsection{Obstruction theory of local rings} We recall other basic notions on deformation theory. We follow the exposition in \cite{sernesi2007deformations} closely.

\begin{definition} Let $S\rightarrow A$ be a morphism of commutative rings and $I$ be an $A$-module. An \textit{extension} of the $S$-algebra $A$ by $I$ consists of an exact sequence of abelian groups \[ 0\rightarrow I \rightarrow A'\rightarrow A \rightarrow 0 \] such that the arrow on the right is a morphism of $S$-algebras such that $I^2 = 0$.
\end{definition}

Two extensions $A'$ and $A''$ are \textit{isomorphic} if there exists an isomorphism $A'\rightarrow A''$ of $S$-algebras that commutes the diagram
\[
\begin{tikzcd}
0 \arrow[r] & I \arrow[r] \arrow[d,equal] & A' \arrow[r] \arrow[d] & A \arrow[r] \arrow[d, equal] & 0 \\
0 \arrow[r] & I \arrow[r] & A'' \arrow[r] & A \arrow[r] & 0.
\end{tikzcd}
\]
We will denote the set of equivalence classes of such extensions by $\Ex(A/S,I)$ which has a natural structure of $A$-module \cite[Section 1.1.2]{sernesi2007deformations}. We define the \textit{pullback} of an extension $0\rightarrow I \rightarrow A'\rightarrow A\rightarrow 0$ by a morphism of $S$-algebras $B\rightarrow A$ as the extension of $B$ by $I$
\[
0\rightarrow I \rightarrow A'\times_A B \rightarrow B \rightarrow 0.
\]
This construction preserves equivalences and induces a function $\Ex(A/S,I)\rightarrow \Ex(B/S,I)$. Furthermore, it can be proved that the set on the right is an $A$-module and that this map is a morphism of $A$-modules.

\begin{proposition}[{{\cite[Prop. 1.1.5]{sernesi2007deformations}}}]{\label{T_sec}} Given a morphism of $S$-algebras $B\rightarrow A$ and an $A$-module $I$ there exists a long exact sequence of $A$-modules 
\[ 
\begin{tikzcd}[row sep = tiny, column sep = small] 
0 \arrow[r] & {\Der(B,I)} \arrow[r] & {\Der(A/S,I)} \arrow[r] & {\Der(B/S,I)} \arrow[r] & {} \\ {} \arrow[r] & {\Ex(A/B,I)} \arrow[r] & {\Ex(A/S,I)} \arrow[r] & {\Ex(B/S,I)}. &
\end{tikzcd}
\]
\end{proposition}

\begin{definition}
Let $S\rightarrow A$ be a morphism between $k$ local Noetherian algebras with residual field $k$. The \textit{obstruction space} of the $S$-algebra $A$ is
\[
o(A/S) := \Ex(A/S,k).
\]
\end{definition}

When $S = k$ we will denote this space simply as $o(A)$. In this local setting, if we specialize Proposition \ref{T_sec} to the case $I = k$ we obtain a long exact sequence of the form
\begin{equation*}
\begin{tikzcd}
0 \arrow[r] & T_{A/B} \arrow[r] & T_{A/S} \arrow[r] & T_{B/S} \arrow[r] & o(A/B) \arrow[r]& o(A/S) \arrow[r] & o(B/S).
\end{tikzcd}
\end{equation*}

\begin{theorem}[{{\cite[Thm 2.1.5]{sernesi2007deformations}}}]{\label{teo_obstrucciones}}
Consider an algebraically closed field $k$ and a morphism $S\rightarrow A$ between $k$-local Noetherian algebras with residual field $k$. If $S$ is a complete $k$-algebra, they are equivalent:
\begin{enumerate}[label = (\alph*)]
\item $A$ is a formally smooth $S$-algebra.
\item $T_{A}\rightarrow T_{S}$ is surjective and $o(A)\rightarrow o(S)$ is injective.
\item $o(A/S) = 0$.
\end{enumerate}
\end{theorem}

Finally, we state a technical result that will be used later.

\begin{lemma}{\label{lema_obstrucciones}}
Let $S\rightarrow B \rightarrow A$ be two morphisms between $k$-local Noetherian algebras with residual field $k$. If $o(A/S) \rightarrow o(B/S)$ is injective, then $o(A)\rightarrow o(B)$ is also injective.
\end{lemma}

\begin{proof} This follows by chasing elements in the commutative diagram
\begin{equation*}
\begin{tikzcd}
0 \arrow[r] & T_{A/S} \arrow[r] \arrow[d] & T_{A} \arrow[r] \arrow[d] & T_{S} \arrow[r] \arrow [equal]{d} & o(A/S) \arrow[r] \arrow[d] & o(A) \arrow[r] \arrow[d] & o(S) \arrow[equal]{d } \\
0 \arrow[r] & T_{B/S} \arrow[r] & T_B \arrow[r] & T_S \arrow[r] & o(B/S) \arrow[r] & o(B) \arrow[r] & o(S).
\end{tikzcd}
\end{equation*}
\end{proof}

\section{Deformations of flags of sheaves}

In this section we present Grothendieck's Drapeaux construction, the moduli space of flags of a given sheaf. We closely follow \cite{drezet1985fibres}, were the local properties of these schemes were stablished. Our focus is on the case of flags of length three.\\

\smallskip

Let $k$ be an algebraically closed field, $X$ a $k$-scheme of finite type and $K$ an upper-bounded complex of $\OO_X$-modules provided with a flag of subcomplexes
\[
0 = F_0K \subseteq F_1K \subseteq \cdots \subseteq F_{\ell-1}K\subseteq F_{\ell}K = K.
\]
In the following we will denote by $\operatorname{gr}_iK = F_iK/F_{i-1}K$ the i-th graded component of $K$. Given another complex $L$ as above, we can construct a new complex $\Hom^\bullet(K,L)$ whose $q$-th homogeneous component is
\[
\Hom^{q}(K,L) := \prod_{i\in\Z}\Hom(K^i,L^{i+q})
\]
and its differential is defined as $d^q(f)^i = d^{q+i}_{Y}\circ f^i + (-1)^qf^{i+1}\circ d_X^i$. The complex $\Hom^\bullet(K,L)$ is equipped with the flag
\[
F_p\Hom(K,L) := \big\{f\in \Hom(K,L) \,:\; f(F_iK) \subseteq F_{i+p}L\;\; \forall i\in\Z\big\}.
\]

On the other hand we define the complex of $\OO_X$-modules
\[
\Hom_+^\bullet(K,L) := \Hom^\bullet(K,L) / F_0\Hom^\bullet(K,L).
\]

\noindent It can be proved that there exists a filtered complex $I$ whose homogeneous components are injective $\OO_X$-modules and that there exists a morphism $L\rightarrow I$ of filtered complexes such that $\operatorname{gr}_i L \rightarrow \operatorname{gr}_i I$ is an injective resolution for all $i\in\Z$.

\begin{definition}
Given a resolution $L\rightarrow I$ as in the previous paragraph, we define the groups
\[
\Ext^i_{+}(K,L) := H^i\big(\Hom_{+}^\bullet(K,I)\big)
\]
which are independent of the chosen resolution.
\end{definition}

\begin{theorem}[{{\cite[p.200]{drezet1985fibres}}}]{\label{spectral_sequence}} There is a spectral sequence with first page  
\[ 
E^{p,q}_1 = \begin{cases} \prod_{i\in \Z}\Ext^{p+q}\big(\operatorname{gr}_{i}(K),\operatorname{gr}_{i-p}(L)\big) \quad & \text{if}\;\; p<0,\\[7pt] 0 \quad &\text{if}\;\; p\geq 0 \end{cases} 
\]
that converges to $\Ext^\bullet_+(K,L)$.
\end{theorem}

\begin{definition}
Let $\X$ be an $S$-projective variety with an $S$-ample line bundle $\OO_\X(1)$. Take a coherent bundle $\E$ on $\X$ that is flat over $S$ and whose Hilbert polynomial associated with $\OO_\X(1)$ is $h$. Given a $\ell$-tuple of rational polynomials $H=(h_1,\dots,h_\ell)$ such that $h = h_1+\cdots + h_\ell$ we define the contravariant functor
\[
\fDrap^{H}_{E} : \operatorname{Sch_{/S}} \rightarrow \operatorname{Set}
\]
whose value at $T\rightarrow S$ is the set of flags of $\OO_{X\times_S T}$-modules
\[
0 = F_0 \subseteq F_1 \subseteq \cdots \subseteq F_{\ell-1}\subseteq F_{\ell} = \E_T = \E\otimes_{\OO_S}\OO_T
\]
that satisfy the following conditions:
\begin{itemize}
\item[$\bullet$] the sheaves $\operatorname{gr}_i\E_T = F_i/F_{i-1}$ are $T$-flat,
\item[$\bullet$] for all $t\in T$, the Hilbert polynomial of $(\operatorname{gr}_i\E_T)|_{X_t}$ is equal to $h_i$.
\end{itemize}
\end{definition}

\begin{remark}
In the particular case where the flags have length $\ell=2$, the functor $\fDrap^{H}_\E$ is just the Grothendieck functor $\fQuot^{h_2}_\E$.
\end{remark}

The functor $\fDrap_{\E}^H$ is represented by an $S$-projective scheme which we will denote by \smash{$\Drap_{\E}^H$} \cite{grothendieck1960techniques}. As usual, it will be convenient to consider the scheme
\[
\Drap^\ell_\E := \bigsqcup_{H} \Drap^{H}_\E
\]
where $H= (h_1,\dots,h_\ell)$ varies among those $\ell$-tuples of polynomials.

In general, if $X\rightarrow S$ is a morphism of schemes and $p\in X$ is a point on $s\in S$, we define the \textit{obstruction space} of $X$ at $p$ relative to $S$ as
\[
o_p(X/S) := o\big(\widehat{\OO}_{X,p}/\widehat{\OO}_{S,s}\big).
\]
Following the results \cite[Prop. 1.5]{drezet1985fibres}, \cite[Prop. 4.4.4]{sernesi2007deformations} and \cite[Prop. 2.1.7]{sernesi2007deformations} we can describe the local properties of the scheme \smash{$\Drap^\ell_{\E}$} as follows:

\begin{proposition}{\label{suc_drap_rel}}
Suppose that $\X$ is a $S$-projective variety and $\E$ is a coherent sheaf on $\X$ that is flat over $S$. Given a flag \smash{$p\in \Drap^\ell_\E$} of the sheaf \smash{$E = \E|_s$} we have an exact sequence
\[
0 \rightarrow \Ext^0_{+}(E,E) \rightarrow T_p\Drap^\ell_{\E} \rightarrow T_sS \rightarrow \Ext^1_{+}(E,E).
\] Additionally, there is an inclusion of the obstruction space \smash{$o_p(\Drap^\ell_\E/S)\subseteq \Ext^1_{+}(E,E)$}. In the particular case $S = \Spec k$, the sequence reduces to an isomorphism \smash{$T_p\Drap^\ell_{E} = \Ext^{0}_{+}(E,E)$}.
\end{proposition}

\subsection{Deformations of flags of length 3} Given a $S$-scheme $T$ and an integer $1\leqslant i\leqslant\ell-1$ there is a map $\varphi_i: \fDrap^\ell_{\E}(T) \rightarrow \fQuot_{\E}(T)$ which assigns to a flag $F_1\subseteq\cdots\subseteq F_{\ell-1} \subseteq \E_T$ the inclusion of the $i$-th subsheaf $F_i\subseteq\E_T$. This defines a natural transformation $\varphi_i: \fDrap^\ell_{\E}\rightarrow \fQuot_{\E}$ between functors of Artin rings and hence a morphism of $S$-schemes
\[
\varphi_i: \Drap^\ell_\E \rightarrow\Quot_\E.
\]

\begin{proposition}{\label{smooth}}
Let $\X$ be a projective $S$-scheme, and $\E$ a coherent sheaf on $\X$ that is flat over $S$. Suppose $p\in \Drap^3_{\E}$ is an element over a smooth point $s\in S$ corresponding to a flag \smash{$F_1\subseteq F_2\subseteq \E|_s = E$} such that
\begin{enumerate}[ label = (\alph*)]
\item $\Hom\big(F_1,E/F_2\big) = 0$, and
\item $o(\varphi_1):o_p\big(\Drap^3_{\E}/S \big)\rightarrow o_{\varphi_1(p)}\big(\Quot_{\E}/S\big)$ is the zero map. \end{enumerate} Then the morphism \smash{$\varphi_2 :\Drap^3_{\E} \rightarrow \Quot_{\E}$} is smooth at $p$.
\end{proposition} 

\begin{proof} By Theorem \ref{teo_obstrucciones} it suffices to show that \smash{$d\varphi_2: T_p\Drap^3_{\E}\rightarrow T_{\varphi_2(p)}\Quot_{\E}$} is surjective and the obstruction map \smash{$o(\varphi_2): o_p(\Drap^3_\E)\rightarrow o_{\varphi_2(p)}(\Quot_\E)$} is injective. We know by Theorem \ref{teo_obstrucciones} that $o_s(S) = 0$ because $s\in S$ is a smooth point. Therefore, using Lemma \ref{lema_obstrucciones}, to prove the injectivity of the map $o(\varphi_2)$ we only need to prove the injectivity of $o(\varphi_2):o_p(\Drap^3_E/S)\rightarrow o_{\varphi_2(p)}(\Quot_E/S)$. Proposition \ref{suc_drap_rel} provides us a commutative diagram with exact rows 

\[ \begin{tikzcd}[column sep = 20 pt, row sep = 18 pt] 0 \arrow[r] & {\Ext^0_{+}\big(E,E\big)} \arrow[r] \arrow[d, "d\varphi_2"] & T_p\Drap^3_{\E} \arrow[r] \arrow[d, "d\varphi_2"] & T_sS \arrow[r] \arrow[equal]{d}& {\Ext^1_{+}(E,E)} \arrow[d, "o(\varphi_2)"] \\ 0 \arrow[r] & {\Hom\big(F_2,E/F_2\big)} \arrow[r] & T_{\varphi_2(p)}\Quot_{\E} \arrow[r] & T_sS \arrow[r] & {\Ext^1\big(F_2,E/F_2\big)} \end{tikzcd} \] 

\noindent and two inclusions \[ o_p\big(\Drap^3_\E/S\big)\subseteq \Ext^1_{+}\big(E,E\big),\quad o_{\varphi_2(p)}\big(\Quot_\E/S\big)\subseteq \Ext^1_{+}\big(F_2,E/F_2\big).
\] 
Chasing elements in this diagram the proof reduces to verifying that the vertical arrow on the left is surjective and that $o(\varphi_2):o_p(\Drap^3_E/S)\rightarrow o_{\varphi_2( p)}(\Quot_E/S)$ is injective.

Theorem \ref{spectral_sequence} applied to the flag $F_1\subseteq F_2\subseteq E$ says that there exists a spectral sequence that converging to $\Ext^\bullet_{+}(E,E)$ such that
\begin{align*}
E_1^{-1,i+1} &= \Ext^i(F_1,F_2/F_1) \oplus \Ext^i(F_2/F_1,E/F_2),\\
E_1^{-2, i+2} &= \Ext^i(F_1,E/F_2),
\end{align*}
and the remaining terms on the first page are zero. As this page is supported only in two consecutive columns, we have a long exact sequence of the form \begin{equation}{\label{sel_espectral}} 
\cdots \rightarrow \Ext^i(F_1,F_2/F_1) \oplus \Ext^i (F_2/F_1,E/F_2) \rightarrow \Ext^{i}_{+}(E,E) \rightarrow \Ext^i(F_1,E/F_2)\rightarrow \cdots \end{equation}

On the other hand the morphisms $\varphi_1,\varphi_2: \Drap^3_{E}\rightarrow \Quot_{E}$ induce a linear applications between tangent spaces
\begin{align*}
 d\varphi_1&: \Ext^0_{+}(E,E) \rightarrow \Hom(F_1,E/F_1)\\
 d\varphi_2&: \Ext^0_{+}(E,E) \rightarrow \Hom(F_2,E/F_2),
\end{align*}
and between obstruction spaces
\begin{align*}
 o(\varphi_1)&: \Ext^1_{+}(E,E) \rightarrow \Ext^1(F_1,E/F_1)\\
 o(\varphi_2)&: \Ext^1_{+}(E,E) \rightarrow \Ext^1(F_2,E/F_2).
\end{align*}
These linear transformations can be coupled to the long exact sequence (\ref{sel_espectral}) to obtain the following commutative diagram:

\hspace{-4.3cm}
{\small
\hspace*{1.25cm}
\begin{tikzcd}[column sep = small]
 & {\Hom(F_1,E/F_1) \oplus \Hom(F_2,E/F_2)} \arrow[rd] & \\
 {\Hom(F_1,F_2/F_1) \oplus \Hom(F_2/F_1,E/F_2)} \arrow[r,hook] \arrow[ru] & {\Ext^{0}_{+}(E ,E)} \arrow[r] \arrow[u] & {\Hom(F_1,E/F_2)} \ar[out=-20, in=160]{dll} \\
 {\Ext^1(F_1,F_2/F_1) \oplus \Ext^1(F_2/F_1,E/F_2)} \arrow[r] \arrow[rd] & {\Ext^{1}_{+} (E,E)} \arrow[r] \arrow[d] & {\Ext^1(F_1,E/F_2) } \\
 & {\Ext^1(F_1,E/F_1) \oplus \Ext^1(F_2,E/F_2)} \arrow[ru] &
\end{tikzcd}
}
\\
\noindent The diagonal arrows involved come from the long exact sequences induced by the bifunctor $\Hom(-,-)$. Using the assumptions of the statement and chasing elements in the previous diagram we can conclude that $\Ext^0_{+}(E,E) \rightarrow \Hom(F_2,E/F_2)$ is surjective and \smash{$o_p\big(\Drap^3_ {\E}/S\big)\subseteq \Ext^1_{+}(E,E)\rightarrow o_{\varphi_2(p)}\big(\Drap^3_{\E}/S\big)\subseteq\Ext^1(F_2,E/F_2)$} is injective as we wanted.
\end{proof}

\begin{remark}{\label{smooth2}}
Applying the same reasoning as in the previous proof, if we replace condition (b) in Proposition \ref{smooth} with the modified condition
\begin{enumerate}
\item[(b')] $o(\varphi_2): o_p\big(\Drap^3_{\E}/S\big) \to o_{\varphi_2(p)}\big(\Quot_{\E}/S\big)$ is the zero map,
\end{enumerate}
we can conclude that $\varphi_1 : \Drap^3_{\E} \to \Quot_{\E}$ is smooth at $p$.
\end{remark}
\section{Deformations of pullbacks under morphisms}

The main goal of this section is to prove Theorem~\ref{stability_morphisms}. Along the way, we establish some rigidity properties of foliations defined by fibrations.

\begin{definition}
Let $\X/S$ and $\Y/S$ be two locally trivial families of normal varieties, $\Pi: \X\rightarrow \Y$ a morphism of $S$-schemes and $\dG/S$ is a family of integrable distributions on $\Y/S$. The \textit{pullback} $\dF/S = \Pi^\ast(\dG/S)$ is defined as the family of integrable distributions on $\X/S$ whose tangent sheaf $T_{\dF/S}$ is the kernel of the composition
\[
T_{\X/S}\rightarrow \Pi^{[\ast]} T_{\Y/S}\rightarrow \Pi^{[\ast]} N_{\dG/S}.
\]
\end{definition}

It should be noted that $\Pi^{[\ast]} N_{\dG/S}$ is torsionless because is the dual of a sheaf, and hence the normal sheaf $N_{\dF/S}$ is also torsionless. In order to prove Theorem \ref{stability_morphisms}, we first need the following vanishing lemma:

\begin{lemma}{\label{vanishing}}
Let $\pi: X\rightarrow Y$ be a proper morphism between normal varieties such that its generic fiber is smooth and has no global holomorphic 1-forms. If $\F = \pi^\ast\G$ is a pullback foliation, then $\Hom(T_{X/Y},N_\F) = 0$.
\end{lemma}

\begin{proof}
Note that since $N_\F$ is torsion-free, then $\sHom(T_{X/Y},N_\F)$ is also torsion-free. So if we want to prove that the global sections of this sheaf are zero, it will be sufficient to verify that such sections are generically zero. This allows us to assume, after restricting ourselves to an appropriate open set of $X$, that the varieties involved, the morphism $\pi$ and the foliation $\G$ are smooth. Thus the normal sheaf of $\F$ turns out to be equal to $N_\F =\pi^\ast N_\G$. Further shrinking the open set if necessary we can assume that $N_\G = \OO_{Y}^{q}$. On the other hand, thanks to the hypothesis about fibers and Proper Base Change Theorem we can conclude that $\pi_\ast\Omega^1_{X/Y} = 0$, therefore $\Hom(T_{X/Y},N_\F) = \Gamma\big(\Omega_{X/Y}^1\big)^q = \Gamma\big(\pi_\ast\Omega^1_{X/Y}\big)^q = 0$.
\end{proof}

If we specialize this lemma to the case where $\G$ is the foliation by points, we obtain the following result on rigidity of fibrations:

\begin{corollary}{\label{rigidity}}
If $\F$ is the foliation whose leaves are the fibers of a proper morphism $\pi: X\rightarrow Y$ with smooth generic fiber and no global 1-forms, then the tangent $T_\F = T_{X/Y}$ is rigid as a subsheaf of $T_X$.
\end{corollary}

%\textcolor{purple}{Conjetura: Si $\pi: X\rightarrow Y$ es suave y $H^0(T_Y) = 0$, entonces $\Hom(T_{X/Y},\pi^\ast T_Y) = 0$. Conozco dos instancias en las que vale esto: si $X$ es un producto, y si $X$ es una fibración elíptica suave (ver mis notas, usa el grado de la imagen directa del anticanónico) }

\subsection{Proof of Theorem \ref{stability_morphisms}} \label{ssec:proofT01}
Observe that the coherent bundle \smash{$\E = T_{\X/S}$} is flat over $S$, hence the choice of an $S$-ample line bundle $\OO_\X(1)$ allows us to construct the schemes \smash{$\Drap^3_\E$} and $\Quot_\E$. Let \smash{$p\in \Drap^3_{\E}$} be the point over $s\in S$ corresponding to the flag
\[
T_{X/Y}\subseteq T_{\F}\subseteq T_{X} = \E|_s.
\]
By Lemma \ref{vanishing}, we know that condition (a) of Proposition \ref{smooth} is verified. In turn, the subsheaf $T_{\X/\Y}\subseteq \E$ corresponds to a section $S\rightarrow \Quot_\E$ of the structural morphism $\Quot_{\E}\rightarrow S$ whose image contains the point $\varphi_1(p)$, and hence condition (b) is a consequence of Corollary \ref{rigidity} and the smoothness of $S$. Similarly, the inclusion $T_{\dF/S}\subseteq \E$ corresponds to a section $\sigma:S\rightarrow \Quot_{\E}$ of the structural morphism such that $\sigma(s) = \varphi_2(p)$. Using \cite[Corollaire 17.16.3]{EGAIV.4} we can find an étale neighborhood $U\rightarrow S$ of $s$ and a lifting 
\begin{center} \begin{tikzcd}[column sep = 25 pt, row sep = 20 pt] 
U \arrow[r, "{\widetilde{\sigma}}"] \arrow[d] & \Drap^3_\E \arrow[d, "\varphi_2"] \\ S\arrow[r, "\sigma"'] & \Quot_\E \end{tikzcd} 
\end{center} such that there exists $u\in U$ with $\widetilde{\sigma}(u) = p$. The morphism $\widetilde{\sigma}:U\rightarrow\Drap_\E$ corresponds to a flag $F_1\subseteq T_{\dF/U} \subseteq T_{\X/U}$. By the rigidity of $T_{X/Y}$ (Corollary \ref{rigidity}) we can conclude, after replacing $U$ by a smaller neighborhood if it is necessary, that $F_1 = T_{\X/\Y}|_U$. The result now follows from Lemma \ref{jorge}.
\section{Stability of algebraic leaves}

We shift our focus to the behavior of algebraic leaves of foliations under deformations and proceed to prove Theorem \ref{stab_leaves}.

\begin{definition}
Suppose that $Z\subseteq X$ is a closed subscheme of a smooth variety $X$. The sheaf of \textit{logarithmic vector fields} along $Z$ is the subsheaf $T_X(-\log Z)\subseteq T_X$ of vector fields witch, acting as derivations of functions, preserve the sheaf of ideals $I_{Z/X}$.
\end{definition}

\begin{definition}
A closed subscheme $Z\subseteq X$ is \textit{invariant} by a foliation $\F$ on $X$ if all the fields tangent to $\F$ preserve the sheaf $I_{Z/X}$, or in other words that
\[
T_\F\subseteq T_X(-\log Z).
\]
If in addition $Z$ is irreducible and $\dim Z = \rank T_\F$ we will say that $Z$ is an \textit{algebraic leaf} of $\F$. Note that in this case $T_\F|_Z$ and $T_Z$ coincide outside the singular points of $Z$ and $\F$.
\end{definition}

If $Z\subseteq X$ is a smooth subvariety, the sheaf $T_X(-\log Z)$ can be framed within the following commutative diagram with exact rows and columns \cite[p. 177]{sernesi2007deformations}: \begin{equation}\label{tang_log} \begin{tikzcd} & 0 \arrow[d] & 0 \arrow[d] & & \\ &I_{Z/X}\,T_X \arrow[d] \arrow[r,equal] & I_{Z/X}\,T_X \arrow[d] & & \\ 0 \arrow[r] & T_X(-\log Z) \arrow[r] \arrow[d] & T_X \arrow[r] \arrow[d] & N_{Z/X} \arrow[r]\arrow[d,equal] & 0 \\ 0 \arrow[r] & T_Z \arrow[r] \arrow[d] & T_X|_Z \arrow[r] \arrow[d] & N_{Z/X} \arrow[r] & 0 \\ & 0 & 0 & &
\end{tikzcd}
\end{equation}

From a smooth subvariety $Z\subseteq X$, we obtain a map
\[
\varphi: \Hilb_{X}\dashrightarrow \Quot_{T_{X}}
\]
defined by the assignment that sends a deformation $\mathcal{Z}\subseteq X\times S$ to the inclusion
\[
T_{X\times S/S}(-\log \mathcal{Z})\subseteq T_{X\times S/S}.
\]
The smoothness of this family follows from the smoothness of both $Z$ and $X$ as the deformations of the inclusion of $Z$ into $X$ are locally trivial. By identifying the normal $N_{Z/X}$ with the quotient $T_X/T_X(-\log Z)$ given by the diagram (\ref{tang_log}) we can define the map
\begin{equation}\label{Lie}
\mathcal{L}:\;N_{Z/X}\rightarrow\sHom\big(T_X(-\log Z), N_{Z/X}\big)
\end{equation}
which assigns to each local section $[v]$ the morphism $[v']\mapsto [\mathcal{L}_v(v')]$ given by the Lie derivative between vector fields. This construction is well-defined because $T_X(-\log Z)\subseteq T_X$ is closed under the Lie bracket. Using \cite[Thm. 1.6]{gomez1988transverse}, the differential of $\varphi$ at the point $Z\in\Hilb_X$ is identified by
\begin{equation}\label{dphi}
d\varphi =\mathcal{L}:\; \Gamma\big(N_{Z/X}\big)\rightarrow\Hom\big(T_X(-\log Z), N_{Z/X}\big).
\end{equation}

\begin{lemma}{\label{lemita}}
Suppose $v$ is a germ of vector field on $X$ such that
\[
\big[v,T_X(-\log Z)\big]\subseteq T_{X}(-\log Z).
\]
Then $v$ belongs to $T_{X}(-\log Z)$. In particular, (\ref{Lie}) is a monomorphism.
\end{lemma}

\begin{proof}
Consider a local coordinate system $x_1,\dots, x_m$ of $X$ such that $Z$ is defined by the equations $x_1 = \cdots = x_r = 0$. For each $1\leqslant i \leqslant r$ we are assuming that $\big[v,x_i \partial_{i}\big]$ is a logarithmic field, and hence the germ function
\[
\big[v,x_i \partial_{i}\big](x_i) = v(x_i) - x_i \partial_i(v(x_i))
\]
vanish along $Z$. Then $v(x_i)$ is also zero in $Z$, and hence $v$ is a logarithmic.
\end{proof}

\subsection{Proof of Theorem \ref{stab_leaves}}\label{ssec:proofT04}
Our goal is to prove that the maps
\[
\varphi:\Hilb_{X}\dashrightarrow \Quot_{T_{X}}\quad \quad \varphi_1: \Drap_{T_{X}}\rightarrow \Quot_{T_{X}}
\]
are smooth at the points corresponding to the inclusion $Z\subseteq X$ and the flag
\[
T_\F\subseteq T_X(-\log Z)\subseteq T_{X}
\]
respectively. Following Remark \ref{smooth2}, the morphism $\varphi_1$ will be smooth in the previous flag if the group $\Hom(T_\F,N_{Z/X})$ is zero and the scheme \smash{$\Quot_{T_X}$} is smooth at the inclusion $T_{X}(-\log Z)\subseteq T_X$. Using that $Z$ is an algebraic leaf of $\F$, we deduce that
\[
\Hom\big(T_\F,N_{Z/X}\big)\simeq \Hom\big(T_\F|_Z,N_{Z/X}\big)\simeq \Hom\big(T_Z,N_{Z/X}\big)\simeq \Gamma\big(\Omega^1_{Z}\otimes N_{Z/X}\big) = 0.
\]
Since $Z$ is an unobstructed subvariety, the proof is reduced to proving only the surjectivity of (\ref{dphi}). Thanks to Lemma \ref{lemita}, this can be further reduced to verifying the inequality
\[
\dim \Hom\big(T_X(-\log Z), N_{Z/X}\big) \leqslant \dim \Gamma\big(N_{Z/X}\big).
\]
From the first column of the diagram (\ref{tang_log}) we obtain the long exact sequence
\[
0\rightarrow \Hom\big(T_Z, N_{Z/X}\big) \rightarrow \Hom\big(T_X(-\log Z), N_{Z/X}\big) \rightarrow \Hom\big(I_{Z/X}\,T_X, N_{Z/X}\big)\rightarrow \cdots
\]
The first of these terms is zero by hypothesis, and therefore
\[
\dim \Hom\big(T_X(-\log Z), N_{Z/X}\big) \leqslant \dim \Hom\big(I_{Z/X}\,T_X, N_{Z/X}\big) = \dim \Gamma\big(\Omega^1_X\big\vert_{Z}\otimes N_{Z/X}^{2}\big).
\]
From what was assumed in item (ii) and the exact conormal sequence
\[
0\rightarrow N_{Z/X}\rightarrow \Omega^1_X\big\vert_{Z}\otimes N_{Z/X}^{2}\rightarrow \Omega^1_Z\otimes N_{Z/X}^{2}\rightarrow 0
\]
we approach the equality \smash{$\dim\Gamma\big(\Omega^1_X\big\vert_{Z}\otimes N_{Z/X}^{2}\big) =\dim \Gamma\big(N_{Z/X}\big)$}. This completes the proof of the theorem.

\begin{remark}
The hypotheses of Theorem \ref{stab_leaves} hold in any of the following cases:
\begin{itemize}
\item $Z$ is a smooth fiber of a submersion $\pi: X\rightarrow C$ with $\Gamma\big(\Omega^1_Z\big) = 0$.
\item $X$ has dimension $\geqslant 3$ and the line bundle $N_{Z/X}^\vee$ is ample.
\item $X$ is a surface and $Z$ is a genus $g$ curve with
\[
Z^2<\min\big\{0,\,2-2g\big\}.
\]
\end{itemize}
In the second and third scenarios the subvariety $Z$ is rigid, since $\Gamma(N_{Z/X}) = 0$. On the contrary, in the first case $Z$ can only be deformed into other neighboring fibers of $\pi: X\rightarrow C$. In the second item we are tacitly using Kodaira-Nakano Vanishing Theorem.
\end{remark}

%\textcolor{purple}{Parace implicar que la parte negativa de una foliación sobre una superficie es estable. En particular la dimensión de Kodaira numérica es invariante por Deformaciones!}

\section{Conical foliation singularities}

The main goal of this section is to prove Theorem \ref{stabcones} on the stability of conical singularities. As mentioned in the introduction, this is a foliated version of Schlessinger's rigidity theorem \cite{schlessinger1973rigid}. To state this result, we first introduce some preliminaries on unfoldings of foliations and a stability theorem due to Camacho and Lins Neto (see Theorem \ref{stab_sing_reg}). In Section~\ref{ssec:conessplit}, we present several examples of rigid singularities that are cones of split foliations.  

\subsection{Unfoldings of foliation germs} \label{ssec:unfoldings}
Throughout this section we will work, unless stated otherwise, on the category of analytic foliation germs on normal complex spaces.

\begin{definition}
Let $\F$ be a foliation germ of codimension $q$ on $(X,x)$. An \textit{unfolding} of $\F$ over a deformation $i:(X,x)\hookrightarrow (\X,x)$ consists of a foliation germ $\dF$ of codimension $q$ over this family such that $i^\ast\dF = \F$.
\end{definition}

Two unfoldings $\dF$ and $\dF'$ over the same family $(\X,x)$ are \textit{equivalent} if there exists a automorphism of deformations $\varphi: (\X,x)\rightarrow (\X,x)$ such that $\varphi^\ast \dF = \dF'$. When the base scheme of the deformation in question is $S=\Spec \C[\varepsilon]/(\varepsilon^2)$ we will say that $\dF$ is a \textit{first order unfolding}. Every unfolding $\dF$ naturally induces a deformation $\dF/S$ of the foliation $\F$. The sheaf $I_{\dF/S}$ that characterizes this family is given by the image of the composition
\[
I_\dF\rightarrow \Omega^{[1]}_\X\rightarrow \Omega^{[1]}_{\X/S}.
\]

If $\dF$ is defined by a $q$-form $\overline{\omega}$, we will denote by $\omega_s = i_s^\ast\overline{\omega}$ the $q$-form that defines the foliation on the fiber $i_s:\X_s\hookrightarrow \X$.

For practical purposes, we will assume from now on that $\F$ has codimension one and is defined on $(\C^n,0)$ by a $1$-form germ $\omega$. For simplicity let $\OO_n$ denote the ring of germs of holomorphic functions of $\C^n$ at the origin, and let $\Omega^p_n$ and $T_n$ be the corresponding $\OO_n$-modules of $p$-forms and vector fields. Finally $\m$ will be the maximal ideal of $\OO_n$. Since $(\C^n,0)$ is rigid, the first-order unfoldings $\dF$ of $\F$ are determined by a differential $1$-form
\[
\overline{\omega} = \omega + \eta \varepsilon + hd\varepsilon
\]
with $\eta\in\Omega^1_n$ and $h\in\OO_n$ verifying the Frobenius condition $\overline{\omega}\wedge d\overline{\omega}=0$. Using the equation $\varepsilon d\varepsilon = 0$, by means of elementary computations, we can conclude that the integrability condition is equivalent to
\[
hd\omega = \omega\wedge(\eta - dh).
\]
The other possible representatives of $\dF$ whose restriction is equal to $\omega$ are of the form
\[
(1+g\varepsilon)\overline{\omega} = \omega + (\eta + g\omega)\varepsilon + h d\varepsilon
\]
with $g\in\OO_n$ invertible. For this reason the first-order unfoldings of $\omega$ are in 1-1 correspondence with the elements of the set
\[
I(\omega) = \big\{h\in\OO_n:\; hd\omega = \omega\wedge\sigma\; \text{for some}\; \sigma\in\Omega^1_n\big\}.
\]
On the other hand, it can be verified that the unfoldings equivalent to the trivial unfolding $\overline{\omega} = \omega$ are identified with
\[
J(\omega) = \big\{h\in\OO_n:\; h = i_v\omega \;\text{for some}\;v\in T_n\big\}.
\]
By contracting both sides of the integrability condition $\omega\wedge d\omega=0$ by any field $v\in T_n$ we can deduce that $J(\omega)\subseteq I(\omega)$. This is summarized in the following statement:

\begin{proposition}[\cite{suwa1992unfoldings}]{\label{inf_unfoldings}}
The equivalence class set of first-order unfoldings of the foliation $\F$ defined by $\omega\in\Omega^1_n$ is in bijection with
\[
\Unf_\omega := I(\omega)/J(\omega).
\]
\end{proposition}
The isomorphism classes of the $\C$-vector spaces $I(\omega)$, $J(\omega)$, and $\Unf_\omega$ are independent of the choice of representative $\omega$ for the foliation $\F$. Therefore, we will adopt the notational convention $\Unf_\F = \Unf_\omega$ when omitting the choice of the $1$-form $\omega$. In the case where $\F$ is a foliation over $\P^n$, we will denote by $\Unf_\F$ the first-order unfolding space of the foliation over $\C^{n+1}$, given by the cone $\pi^\ast \F$, where $\pi: \C^{n+1} \dashrightarrow \P^n$ is the projection to the quotient.

The members of the modules $\OO_n$, $\Omega^p_n$ and $T_n$ are formally decomposed as a sum of graded elements. This grading is determined by the identities $\deg x_i = \deg dx_i = -\deg\partial/\partial x_i = 1$ for $1\leqslant i \leqslant n$. Let us denote by $\OO_n(\ell)$, $\Omega^p_n(\ell)$ and $T_n(\ell)$ the respective homogeneous components of degree $\ell\in\Z$. When $\omega$ is homogeneous of degree $k$ every member of the spaces $I(\omega)$, $J(\omega)$ and $\Unf_\omega$ can be formally decomposed as a sum of their graded components $I(\omega)(\ell)$, $J(\omega)(\ell)$ and $\Unf_\omega(\ell)$ respectively. On the other hand, let us recall that an \textit{integrating factor} of $\omega$ is a germ $h\in\OO_n$ such that $d(\omega/h) = 0$. Equivalently, the integrating factors of $\omega$ belongs to
\[
K(\omega) = \big\{h\in\OO_n:\; hd\omega = dh\wedge\omega\big\}.
\]
Clearly $K(\omega)\subseteq I(\omega)$. For what follows we will need the following intermediate sets
\[
I^{(k)}(\omega) = \big\{h\in\OO_n:\; hd\omega = \omega\wedge(\eta - dh)\; \text{for some}\; \eta\in\m^{k-1}\Omega^1_n\big\}.
\]

\begin{definition}
Let $\F$ be a foliation germ of codimension one over $(\C^n,0)$ induced by a holomorphic $1$-form $\omega$.
\begin{itemize}
\item[$\bullet$] An unfolding $\dF$ of $\F$ over $(\C^{n+m},0)$ is \textit{$k$-trivial} if $j^k\omega_s = j^k\omega$ for all $|s|\ll 1$.
\item[$\bullet$] $\F$ is \textit{locally $k$-determined} if for all $k$-trivial unfolding $\dF$ of $\F$ over $(\C^{n+m},0)$ the 1-forms $\omega_s$ and $\omega$ induce isomorphic foliations for all $|s|\ll 1$.
 \item[$\bullet$] $\F$ is \textit{infinitesimally $k$-determined} if $I^{(k+1)}(\omega) \subseteq \m J(\omega) + K(\omega)$.
\end{itemize} 
\end{definition} 

\begin{theorem}[{{\cite[p. 991]{suwa1985determinacy}}}]{\label{determinacy_suwa}}
Let $\F$ be a codimension one foliation germ over $(\C^n,0)$ induced by a $1$-form $\omega$ such that
\[
\dim K(\omega)/\m J(\omega) \cap K(\omega) < \infty.
\]
If $\F$ is infinitesimally $k$-determined, then it is locally $k$-determined.
\end{theorem}

In the framework of classical Singularity Theory a result of Tougeron states that every isolated singularity germ $f\in\OO_n$ with Milnor number $\mu = \mu(f)$ is $(\mu+1)$-determined, i.e. $f$ is equivalent to its $(\mu+1)$-jet \cite{arnold_singularities_vol1}. Using Suwa's Theorem one can establish the following criterion of finite determinacy for foliation germs of the same nature:

\begin{proposition}{\label{determinacy_los_chicos}}
Let $\F$ be a foliation germ of codimension one on $(\C^n,0)$ induced by $\omega$. Suppose that $j^{k-1}\omega = 0$ and $\omega_k= j^k\omega$ is a nonzero homogeneous $1$-form without polynomial integrating factors satisfying
\[
\Unf_{\omega_k}(\ell) = 0\quad\; \forall\;\ell > k.
\]
Then $\omega$ and $\omega_k$ induce isomorphic foliations.
\end{proposition}

This proposition is a mere adaptation of a theorem of Molinuevo and Quallbrunn \cite{molinuevo2017camacho}. As in the proof of Tougeron's Theorem, the authors' strategy (and ours as well) consists in applying Thom's \textit{homotopy method}. Specifically, the expression
\[
\overline{\omega}(x,s) = \sum_{\ell\geqslant k}s^{\ell-k}\,\omega_{\ell}(x)
\]
determines a family that interpolates the forms $\overline{\omega}(x,0) = \omega(x)$ and $\overline{\omega}(x,1) = \omega_k(x)$. This is an unfolding since $s^k\overline{\omega}(x,s) = \omega(sx)$ is clearly integrable, which is also $k$-trivial (as an unfolding of the germ $\omega_s$ for all $0\leqslant s\leqslant 1$). If we succeed in justifying that $\omega_s$ is locally $k$-determined, for any pair of values of $0\leqslant s\leqslant 1$, the foliations that induce these forms must be isomorphic.

\begin{proof}[Proof of Proposition \ref{determinacy_los_chicos}]
By Theorem \ref{determinacy_suwa} it suffices to prove that the foliation $\F$ is infinitesimally $k$-determined. Given a germ $h\in I^{(k+1)}(\omega_k)$, by definition there must exist a $1$-form $\eta\in \m^k\Omega^1_n$ such that
\[
hd\omega_k = \omega_k\wedge(\eta-dh).
\]
Let \smash{$h = \sum_{\ell\geqslant 0}h_\ell$} and \smash{$\eta = \sum_{\ell\geqslant 0}\eta_\ell$} be the decompositions into homogeneous components of these elements. Since $\omega_k$ is homogeneous, the integrability condition in the previous equation leads to the identities
\[
h_\ell d\omega_k = \omega_k\wedge(\eta_\ell-dh_\ell)
\]
with $\ell\geqslant 0$. The condition on $\eta$ implies that $\eta_\ell = 0$ for all $\ell\leqslant k$ (the $1$-forms have positive degree) and hence $h_\ell = 0$ for the same collection of indices (since $\omega_k$ has no polynomial integrating factors). For any remaining value $\ell> k$ the condition $\Unf_{\omega_k}(\ell) = 0$ implies the existence of a homogeneous field $v_{\ell-k}\in T_n$ of degree $\ell-k$ such that $h_\ell = i_{v_{\ell-k}}\omega_k$. Consequently the field $v = \sum_{\ell>k}v_{\ell-k}$ vanishes at the origin and is a formal solution of the equation $h = i_v\omega_k$. By Artin Approximation Theorem \cite{artin1968solutions} there must exist an analytic solution belonging to $\m T_n$. Then, $h\in \m J(\omega_k) = \m J(\omega_k) + K(\omega_k)$ as claimed.
\end{proof}

\begin{remark}{\label{symmetries}}
A homogeneous $1$-form $\omega_k\in\Omega^1_n$ of degree $k$ without polynomial integrating factors such as the one appearing in the statement of Proposition \ref{determinacy_los_chicos} must necessarily descend to projective space. More generally: if the field $v\in T_n$ is a \textit{symmetry} of a foliation germ $\omega\in \Omega^1_n$, i.e.
\[
L_v\omega\wedge\omega =0,
\]
then $h = i_v\omega$ is a polynomial integrating factor of $\omega$. Indeed,
\begin{align*}
dh\wedge\omega = (L_v\omega - i_vd\omega)\wedge \omega = - i_vd\omega\wedge \omega = hd\omega.
\end{align*}
The Euler radial field
\[
R = x_1\frac{\partial}{\partial x_1}+\cdots +x_n\frac{\partial}{\partial x_n}
\]is a symmetry for any homogeneous $1$-form since after a simple computation the classical identity $L_R\omega_k = k \omega_k$ is obtained. Going even further, such homogeneous 1-forms without polynomial integrating factors are exactly those that do not belong to any of the logarithmic components of the moduli space of foliations over $\P^n$ (see \cite{degree3}).
\end{remark}

\subsection{Camacho-Lins Neto regularity} \label{ssec:CLNreg}

As we have seen in the first section, there are two families of stable singularities of foliations of codimension one: the Morse-type and the Kupka-type singularities. Next we will add one more family to the list, those called \textit{regular} singularities. They were introduced by Camacho and Lins Neto in \cite{camacho1982topology}. Later Molinuevo discovered that the regularity condition is intimately related to the first-order unfoldings of the foliation itself \cite{molinuevo2016unfoldings}.

Suppose that $\omega\in\Omega^1_{n}$ is an integrable homogeneous $1$-form of degree $k$. Consider the complex of $\C$-vector spaces
\[
C^\bullet_\omega:\quad T_n\rightarrow \Omega^1_n\rightarrow \Omega^3_n
\]
starting at degree zero, and whose differentials are
\begin{align*}
d^0(v) & = L_v\omega,\\
d^1(\eta) & = \eta\wedge d\omega + \omega\wedge d\eta.
\end{align*}
Note that the composition of both morphisms is zero since $d^1\circ d^0 = L_v(\omega\wedge d\omega)$. By homogeneity of $\omega$, for each $\ell\in\Z$ we obtain a subcomplex
\[
C^\bullet_\omega(\ell):\quad T_n(\ell-k)\rightarrow \Omega^1_n(\ell)\rightarrow \Omega^3_n(\ell+k).
\]

\begin{definition}
We say that $\omega$ is \textit{regular}, in the Camacho-Lins Neto sense, if
\[
H^1\big(C^\bullet_\omega(\ell)\big) = 0\quad\; \forall\;\ell < k.
\]
The \textit{rank} of $\omega$ is the dimension of the space of $1$-coboundaries of the complex $C^\bullet_\omega(k-1)$, that is, the dimension of the image of the differential $d^0:T_n(-1)\rightarrow \Omega^1_n(k-1)$.
\end{definition}

\begin{theorem}[{{\cite[p. 1600 and 1608]{molinuevo2016unfoldings}}}]
Given an integrable $1$-form $\smash{\omega\in\Gamma\big(\Omega^1_{\P^n}(k)\big)}$, there exist isomorphisms of $\C$-vector spaces
\[
\Unf_\omega(\ell) \simeq H^1\big(C^\bullet_\omega(\ell)\big)
\]
for all $\ell\neq k$. In particular, $\omega$ is regular if and only if $\Unf_{\omega}(\ell) = 0$ for all $\ell < k$.
\end{theorem}

This Theorem is obtained by combining Proposition 3.1.5, Corollary 6.1.8 and Theorem 3.2.2 of the cited article. This pair of spaces are not necessarily isomorphic if $\ell = k$. The elements of the first four families of foliations in Table 2 of \cite{degree3} satisfy that $\Unf_\F(k) = 0 \neq H^1\big(C^\bullet_\omega(k)\big)$ (as we will see later, by Proposition \ref{split}, $\Unf_\F = 0$ is verified but these foliations are not rigid).

\begin{definition}
Let $\F$ be a foliation of codimension one on $X$. We will say that a singularity $p$ of $\F$ is \textit{regular of degree $k$ and rank $r$} if $X$ is smooth in $p$ and there exists a $1$-form germ $\omega\in\Omega^1_{X,p}$ defining $\F$ whose k-jet $j^k_p\omega$ is homogeneous and regular of rank $r$.
\end{definition}

\begin{theorem}[{{\cite[p. 8]{camacho1982topology}}}]{\label{stab_sing_reg}}
Let $\F$ be a foliation of codimension one on $X$.
\begin{enumerate}[label = (\alph*)]
\item The set $M_k^r(\F)$ of regular singularities of degree $k$ and rank $r$ of $\F$ is,assuming it is nonempty, a smooth subscheme of $X$ of codimension $r$.
\item Given a deformation $\dF/S$ of $\F$ as an involutive Pfaff system on $\X/S$ and smooth base $S$, the set $M_k^r(\dF)=\bigcup_{s\in S} M_k^r(\dF_s)$ is smooth on $S$.
\end{enumerate}
\end{theorem}

\subsection{Stability of conical singularities}\label{ssec:conical}

\begin{definition}
Let $\F$ be a foliation germ on $(\C^n,0)$. An analytical deformation $\dF/\C^{m}$ of $\F$ on $(\C^{n+m},0)$ is said \textbf{versal} if for any other deformation $\dF'/\C^{m'}$ of $\F$ on $(\C^{n+m'},0)$ there exists a commutative diagram 
\begin{center} 
\begin{tikzcd}[column sep = 0.5 em] & {(\C^n,0)} \arrow[ld, hook] \arrow[rd, hook] & \\ {(\C^{n+m'},0)} \arrow[rr,"\varphi"] \arrow[d] & & {(\C^{n+m},0)} \arrow[d] \\ {(\C^{m'},0)} \arrow[rr] & & {(\C^{m},0)}
\end{tikzcd}
\end{center}
such that $\varphi^\ast(\dF/\C^{m})\simeq \dF'/\C^{m'}$.
\end{definition}

The moduli space of foliations of codimension one and degree $k$ over the projective space $\P^n$ is the scheme
\[
\Fol_{\P^n,k} = \Big\{[\omega]\in \P\big(\Gamma(\Omega^1_{\P^n}(k+2))\big):\,\, \omega\wedge d\omega = 0,\,\, \codim S(\omega)\geqslant 2\Big\}.
\]
This has attached to it a universal family of foliations of this type with base $S=\Fol_{\P^n,k}$ which we will simply denote as $\dF/S$.

\subsection{Proof of Theorem \ref{stabcones}}\label{ssec:proofT03}
Let us take a deformation $\dF'/\C^{m'}$ of the foliation $\F$. By Theorem \ref{stab_sing_reg} we can assume without loss of generality that $\dF'_s$ has a regular singularity at the origin for all $|s|\ll 1$. By the semi-continuity of the space dimension $K(\omega)$ in families, the foliations $\dF'_s$ will also have no polynomial integrating factors. 
Proposition~\ref{determinacy_los_chicos} and Remark \ref{symmetries} allow us to further assume that \smash{$\dF'/\C^{m'}$} descends to the projective space $\P^n$. The rest of the proof follows from the versality of the deformation~$\dF/S$.

\begin{remark}
It is worth highlighting the connection between the Theorem \ref{stabcones} and the work of Cerveau and Mattei on moduli spaces of foliations over $\C^{n+1}$ determined by homogeneous $1$-forms of degree $k$ prefixed \cite[Quatrième partie]{mattei1982d}. They proved that the irreducible components of such a space are the logarithmic and non-logarithmic components of $\Fol_{\P^n,k}$. Note that the deformations of analytic germs considered in this section are not necessarily polynomial, which is why neither of these two results implies the other.
\end{remark}

\subsection{Cones of foliations with split tangent sheaf}\label{ssec:conessplit}

The main source of examples of foliations $\F$ that verify the hypotheses of Theorem \ref{stabcones} are those of \textit{split} type, that is, those whose tangent sheaf $T_\F$ is a direct sum of line bundles.

\begin{proposition}[{{\cite[p. 21]{massri2015kupka}}}]{\label{split}}
Let $\mathscr{F}$ be a split foliation of codimension one on $\P^n$ such that $\Sing(d\omega)$ has codimension greater than or equal to $3$. Then $\Unf_\mathscr{F} = 0$.
\end{proposition}

\noindent Let us enumerate families of examples that verify the hypotheses of Theorem \ref{stabcones}:

\begin{example}
A very general foliation by curves of degree $k\geqslant 2$ over $\P^2$ has only Kupka-type singularities \cite{neto1988algebraic} and no polynomial integrating factors (in fact, these foliations do not have algebraic leaves \cite[p. 158]{jouanolou2006equations}).
\end{example}

\begin{example}{\label{TM}}
The next family of examples of foliations on $\P^3$ tangent to a multiplicative action is extracted from~\cite[p.10]{degree3}. These are the generic elements of certain closed subsets $TM_d(a,b,c;n)\subseteq \Fol_{\P^3,d}$ that we will describe below. Let $a,b,c,n$ be non-negative integers such that $0\leqslant a < b < c$  do not have common divisors. Consider the vector field
\[
v_{(a,b,c)} = ax_0\frac{\partial}{\partial x_0} + bx_1\frac{\partial}{\partial x_1} + c x_2\frac{\partial}{\partial x_2},
\]
described in homogeneous coordinates. The elements of $TM_d(a,b,c;n)$ induce those foliations on $\P^3$ of degree $d$ which up to a linear change of coordinates are determined by a $1$-form $\omega$ satisfying the conditions
\[
i_{v_{(a,b,c)}}\omega = 0, \quad L_{v_{(a,b,c)}} \omega = n\omega.
\]
Assuming that $d\geqslant 3$ and $1\leqslant a<b<c$, \cite[Thm 4.12]{degree3} presents a characterization in terms of the integer parameters involved of the irreducible components of the form $TM_d(a,b,c;n)$ whose generic element is split and has at most finite non-Kupka singularities. On the other hand, Proposition 4.5 of the same article establishes a criterion to determine whether these foliations have a polynomial integrating factor. They give in their Table 2 an exhaustive list with all the possible components that verify these two conditions, including the degenerate cases $0\leqslant a\leqslant b\leqslant c \neq 0$.
\end{example}

\begin{example}
One of the 6 irreducible components of the moduli space $\Fol_{\P^n,2}$ with $n\geqslant 3$ is the well known \textit{exceptional component} $\operatorname{E}(n)$ \cite[p. 580]{cerveau1996irreducible}. In order to describe it we will consider the action of the affine group
\[
\operatorname{Aff(\C)} = \left\{
\begin{pmatrix}
a & b \\
0 & a^{-1}
\end{pmatrix} \Bigg|\;\; a\in\C^{\ast},\;\; b\in \C
\right\}
\]
on the space $\P^3 = \P(\Sym^3\C^2\big)$ of binary forms of degree 3 via changes of coordinates. The affine Lie algebra $\mathfrak{aff(\C)}$ is generated by the matrices
\[
X = \begin{pmatrix}
1 & 0 \\
0 & -1
\end{pmatrix}, \quad Y = \begin{pmatrix}
0 & 1 \\
0 & 0
\end{pmatrix}
\]
which act on the basis $z_i = x^{3-i}y^i$ of $\Sym^3\C^2$ as
\[
X\cdot z_i = (3-2i) z_i, \quad Y\cdot z_i = z_{i+1}.
\]
They are therefore expressed in these coordinates as
\[
X = \sum_{i=0}^3(3-2i)z_i\frac{\partial}{\partial z_i}, \quad Y = \sum_{i=0}^{2}z_{i+1}\frac{\partial}{\partial z_{i}}.
\]
These vector fields induce a split foliation $\F$ of codimension one over $\P^3$ which has a unique non-Kupka singularity \cite[p.10]{calvo2006note}. The component $E(n)\subseteq \Fol_{\P^n,2}$ is the closure of the set of foliations of the form $\pi^\ast \F$, with $\pi: \P^n\dashrightarrow \P^3$ a linear projection.
\end{example}

\begin{example}
More generally, consider Lie algebras generated by the following vector fields over $\P^n$
\[
X = \sum_{i=0}^n(n-2i)z_i\frac{\partial}{\partial z_i},\quad Y_j = \sum_{i=0}^{n-k}z_{i+j}\frac{\partial}{\partial z_{i}}\quad j = 1,\dots,n-2.
\]
\noindent Taking linear pullbacks of the foliation generated by them, new irreducible components generalizing the exceptional components were obtained \cite[p.5]{CukiermanSplit}. They also found two Lie algebras $\mathfrak{g}_6$ and $\mathfrak{g}_7$ which determine rigid foliations of codimension one on $\P^6$ and $\P^7$ respectively \cite[Prop. 6.9]{CukiermanSplit}.
\end{example}
\section{Deformation of pullback foliations under rational maps}

In this section we introduce the genericity condition between a rational map~$\pi:X\dashrightarrow Y$ and a foliation $\G$ on $Y$ which appears among the hypotheses of Theorem~\ref{stability_rational_maps}. We also provide a proof of this result.

\begin{definition} Let $\pi:X\dashrightarrow Y$ be a rational map between normal varietiess, $\G$ a foliation on $Y$, and $\F = \pi^\ast\G$ the pullback between the two. We define the \textit{scheme of tangents} between $\pi$ and $\G$ as the schematic difference \[ \Tang\big(\pi,\G\big) = \Sing \F\setminus\pi^{-1}(\Sing \G).
\] \end{definition}

\begin{definition}
Let $X$ be a normal projective variety and $\pi:X\dashrightarrow \P^n$ a dominant rational map determined by an ample line bundle $L$ over X and sections $s_0, \dots, s_n\in\Gamma(L)$. We will say that $\pi$ \textit{generic} if
\begin{enumerate}[label =(\alph*)]
\item $X$ is smooth along all points of the base locus $B$,
\item the sections $s_i$ intersect transversely along $B$.
\end{enumerate} 
\end{definition} 

\begin{lemma}{\label{base_locus_structure}} Given a point $p\in B$ at the base locus of a generic rational map $\pi: X\dashrightarrow \P^n$, there is an holomorphic coordinate system $z_1,\dots,z_m$ centered at $p$ such that 
\[
\pi(z_1,\dots,z_{m}) = [z_1: \ldots :z_{n+1}].
\]
\end{lemma}

\begin{proof}
It suffices to simply take a sufficiently small analytic neighborhood $U$ of $p$ such that $L|_U\simeq\OO_U$ and extend the sections $s_0,\dots,s_n\in\Gamma(\OO_U)$ to a coordinate system (note that this is possible thanks to both genericity conditions).
\end{proof}

For this class of maps, the topology of the domain $X$ imposes topological constraints on the fibers of $\pi$ as proved by the following \say{Lefschetz-type} result:

\begin{proposition}{\label{lefschetz}}
Suppose $X$ is a complete local intersection of dimension $m$ and $\pi: X\dashrightarrow \P^n$ is a generic rational map such that $m-n\geqslant 2$ and $H^1(X,\C)=0$. Then every fiber $F$ of $\pi$ is connected and also satisfies $H^1(F,\C) = 0$.
\end{proposition}

\begin{proof}
The basis of the proof is to apply the Fulton-Lazarsfeld version of Lefschetz's Hyperplane Sections Theorem \cite[p. 28]{fulton2006connectivity}. To do so we will need to prove that we can embed $X$ in a sufficiently large projective space $\P^N$ such that $F$ is the set-theoretical intersection of $X$ with $n$ hyperplanes. If this were the case, the above theorem in conjunction with the hypothesis $m-n\geqslant 2$ would guarantee us that the pair $(X,F)$ is $2$-connected. By the Relative Hurewicz Theorem this implies that $H_k(X,F) = 0$ for $0\leqslant k \leqslant 2$, and finally $H^1(F,\C)= 0$ by the Universal Coefficients Theorem.

Let $F$ be the fiber of $\pi$ over the point $p\in\P^n$. Let $h_i\in \Gamma\big(\OO_{\P^n}(1)\big)$ with $0\leqslant i\leqslant n$ such that $p$ is the unique solution of the system $f_0=\cdots = f_n = 0$. Replacing the sections $f_i$ by convenient powers of them if necessary, we can assume that the degrees $d_i$ are all equal to an integer $d$. For the same reason, we can assume that $\pi^\ast\OO_{\P^n}(d)$ is very ample. This line bundle defines the desired embedding because $F$ is set-theoretically described as the solutions of the system $\pi^\ast f_0=\cdots = \pi^\ast f_n = 0$.
\end{proof}

\subsection{Stability of pullback foliations}

\begin{definition}{\label{generic_pair}}
Let $X$ be a normal projective variety, $\pi:X\dashrightarrow \P^n$ a dominating rational map, and $\G$ a foliation of codimension one over $\P^n$. We will say that the pair $(\pi,\G)$ is \textit{generic} if
\begin{enumerate}[label =(\alph*)]
\item the rational map $\pi$ is generic,
%\item If $NF\subseteq X$ is the non flat locus of $\pi$, we want $\pi(NF)\cap \Sing_\G = \emptyset$.
\item the regular values of $\pi$ are dense in $\Sing(\G)$,
%\item The multiplicity $\ell:\Fol_{\P^n}\rightarrow\Z_{\geqslant 0}$ is locally constant around $\G\in\Fol_{\P^n}$.
%\item \textcolor{red}{The singularities of $\G$ are stable.}
\item the tangential scheme $\Tang(\pi,\G)$ consists of at worst Morse singularities.
\end{enumerate}
\end{definition}

In the next section we will see that if $X$ is smooth, $\pi$ is a generic morphism and $\G$ is a foliation over $\P^n$ then there exists a nonempty open set $U\subseteq \Aut(\P^n)$ such that the pair $(\pi,\sigma^\ast \G)$ is a generic pair for all $\sigma\in U$ (see Theorem \ref{finitasTangencias}).

\subsection{Proof of Theorem \ref{stability_rational_maps}} \label{ssec:proofT02}
For simplicity we will denote by $\X/S$ the trivial family $X\times S/S$. Due to Lemma \ref{base_locus_structure}, Proposition \ref{split} and Theorem \ref{stabcones} the singularities of the foliation $\F =\pi^\ast\G$ along the base locus $B$ are stable. After replacing $S$ by a smaller open set, we obtain a deformation $\B\subseteq \X$ of the embedding $B\subseteq X$ such that $\B\subseteq\Sing(\dF/S)$. Shrinking $S$ once again, there must exist deformations $\sigma_0,\ldots,\sigma_n\in\Gamma(L\otimes\OO_S)$ of the sections $s_0,\ldots,s_n$ respectively such that $\B = V(\sigma_0,\ldots,\sigma_n)$ (see \cite[Teorema 1.2.2]{perrella2024}). Thus we obtain a rational map $\Pi: \X \dashrightarrow\P^n$ defined as $\Pi(x,s)=[\sigma_0(x,s),\ldots,\sigma_n(x,s)]$.

On the other hand, consider the blow-up $\Bl \X$ of the family $\X$ along $\B$. By pulling back the deformation $\dF/S$ by the $S$-morphism $\Bl\X\rightarrow \X$ we obtain a deformation $\Bl(\dF)/S$ of the pullback of $\G$ by the map $\Bl X\rightarrow\P^n$. This follows from the combination of Lemma \ref{def_blowup}. Replacing $S$ by an open set we can assume that $\Sing(\dG/S)$ is a flat family of subschemes of $\P^n$, and by the local structure of Lemma \ref{base_locus_structure} we conclude that $\Sing(\Bl(\dF)/S)$ is flat in a neighborhood of the exceptional divisor of $\Bl\X$. Recall that the singular locus of $\F$ decomposes as
\[
\Sing\F = \Tang(\pi,\G)\cup \pi^{-1}(\Sing \G).
\]
Following the proof of \cite[Thm 9.2]{Quallbrunn_2015}, the singular locus of $\G$ is equal to the closure of the Kupka singularities of the foliation. The second of the condition for generic pairs $(\pi,\G)$ and hypothesis (d) on the singularities of $\G$ imply that $\pi^{-1}(\Sing \G)$ is in turn the closure of the Kupka-type singularities of $\F$. Since the Morse and Kupka singularities are stable, then $\Sing(\Bl(\dF)/S)$ is flat over $S$ (\cite[Lemma 8.12]{Quallbrunn_2015}). Lemma \ref{dualizar} guarantees that the dual family to $\Bl(\dF)/S$ is a deformation of the pullback foliation of $\G$ by $\Bl X\rightarrow\P^n$.

Note that Hodge Decomposition Theorem and \ref{lefschetz} Proposition imply that the generic fiber of $\Bl X\rightarrow\P^n$ is connected and has no non-zero global $1$-forms. Finally, Theorem \ref{stability_morphisms} implies the existence of an analytic neighborhood $U$ such that $\Bl(\dF)/U$ is the pullback of $\dG/U$ by $\Bl\X_U\rightarrow\P^n$ and hence $\dF/U=\Pi^\ast(\dG/U)$ as claimed.

\begin{remark}
It remains open whether this proof scheme can prove stability of pullback foliations via rational maps $\pi: X\dashrightarrow \P_w$ reaching weighted projective spaces. Most of the ideas work in a similar fashion, but an extension of Theorem \ref{stabcones} for \textit{weighted conical singularities} is needed. Those that are pullbacks of a foliation on $\P_w$ via the projection to the quotient map \smash{$\C^{n+1}\dashrightarrow\P_w$}. Some particular cases of stability behavior are known, such as those presented in \cite{calvo2004irreducible} or certain foliation cones in $TM_{d}(a,b,c;n)$, as illustrated in Example \ref{TM}, which are pullbacks of rational maps $\P^3 \dashrightarrow \P(a,b,c)$ \cite{degree3}. This raises the possibility of a generalization of Theorem \ref{stabcones}.
\end{remark}

\section{Tangential singularities}

Generic pairs $(\pi,\G)$ have only finitely many tangential singularities, as stated in condition (c) of Definition \ref{generic_pair}, and all of them are Morse singularities. We apply Kleiman's transversality theorem to prove that this condition is in fact generic. This is established in Theorem \ref{finitasTangencias}.\\

Let $X$ and $Y$ be two smooth varieties of dimensions $m$ and $n$ respectively. Given a morphism $\pi: X\rightarrow Y$ and an integer $0\leqslant k\leqslant \min\{m,n\}$ the \textit{$k$-th critical set} is defined as
\[
C_k(\pi) = \big\{p\in X\, :\, \rank (d_p\pi) \leqslant k\big\}.
\]
Locally around trivializing open sets, the morphism $d\pi: T_X\rightarrow \pi^\ast T_Y$ is given by a matrix of size $m\times n$. $C_k(\pi)$ is the set where the $(k+1)\times (k+1)$ minors of such matrices vanish, and those equations provides it a scheme structure. In general, if $\pi: X\dashrightarrow Y$ is a rational map with maximal domain $U\subseteq X$, we define its \textit{$k$-th critical set} as the subscheme $C_k(\pi):= C_k(\pi|_U)$.

Assuming it is non-empty, $C_k(\pi)$ has codimension less than or equal to $(m-k)(n-k)$ as occurs with every determinantal variety \cite{arbarello1985geometry}. The number on the right-hand side is referred to as the \textit{expected codimension}. In this context, we provide a bound on the dimension of $C_k(\pi)$, with the proof inspired by \cite[Thm. 3]{cukierman2006singularities} and \cite[Prop. 4.4]{molinuevo2022singular}.

\begin{proposition}{\label{expecteddimension}}
Let $\pi: X\dashrightarrow\P^n$ be a generic rational map. If the critical set $C_k(\pi)$ is nonempty, then
\[
m-(m-k)(n-k) \leqslant \dim C_k(\pi) \leqslant k.
\]
\end{proposition}

\begin{proof}
By what we discussed above, the first inequality holds in general. We are going to argue by induction on $n\geqslant 0$. Since the base case $n=0$ is obvious, we proceed to assume that $n\geqslant 1$. Consider a generic hyperplane $\P^{n-1}\subseteq\P^n$. By Bertini's Theorem, the preimage $X' = \overline{(\pi|_U)^{-1}(\P^{n-1})}$ is smooth outside the base locus $B$. We further claim that the divisor $X'$ is smooth along $B$ as well. Suppose instead that $X'$ is singular at a point $p\in B$, and take an equation $a_0x_0+\cdots+a_nx_n = 0$ determining the hyperplane $\P^{n-1}$. By means of a local trivialization $L|_U\simeq \OO_U$ around $p$ the sections $s_i=\pi^\ast x_i\in\Gamma(L|_U)$ are identified with regular functions on $U$, which we will denote by $s_i$ as well. Since $p\in X'$ is singular we have that $a_0d_ps_0+\cdots+a_nd_ps_n = 0$, and hence $d_ps_0\wedge\cdots\wedge d_ps_n = 0$. But this in turn contradicts the transversality hypothesis of the definition of generic maps.

We now note that the restriction $\pi|_{X'}: X'\dashrightarrow\P^{n-1}$ is a generic rational map determined by the ample line bundle $L|_X$, and by inductive hypothesis each of its critical sets is either empty or has codimension equal to the expected one. Moreover, it is clear that $C_k(\pi)\cap X'\subseteq C_{k-1}(\pi|_{X'})$ (granting that $C_0(\pi)\cap X'\subseteq C_{0}(\pi|_{X'})$). Replacing the divisor $X'$ if necessary, we may assume that $\dim \big(C_k(\pi)\cap X'\big) = \dim C_k(\pi) -1$ (set $\dim \emptyset = -1$). Then $\dim C_k(\pi) \leqslant \dim \big(C_k(\pi)\cap X'\big)+1 \leqslant \dim C_{k-1}\big(\pi|_{X'}\big)+1 \leqslant k$.
\end{proof}

As previously mentioned, the structure of the tangency scheme $\Tang(\pi,\G)$ between a map and a foliation is easy to describe:

\begin{theorem}{\label{finitasTangencias}}
Let $\pi: X\dashrightarrow \P^n$ be a generic rational map and $\G$ a codimension one foliation in $\P^n$. Then there exists a nonempty open set $U\subseteq \operatorname{\Aut(\P^n)}$ such that, for all $\sigma\in U$, the tangency scheme $T(\pi,\sigma^\ast\G)$ consists of at most finite Morse-type singularities of the foliation $\F_\sigma = \pi^\ast\sigma^\ast\G$.
\end{theorem}

\begin{proof} To simplify notation let $C_k^0(\pi)$ be the difference $C_k(\pi)\setminus C_{k-1}(\pi)$, $Z$ be the singular locus of $\G$, and $G = G(k+1,n+1)$ be the Grassmannian of $k+1$-dimensional subspaces of $\C^{n+1}$. Consider the diagram 
\[ 
\begin{tikzcd}[row sep =25 pt] & {(\P^n\setminus Z)\times G} \arrow[d] \\ C_k^0(\pi)\times (\P^n)^\vee \arrow[r] & {\P^n\times G\times (\P^n)^\vee} \end{tikzcd} 
\] 
which has as a horizontal morphism $( p,H)\mapsto (\pi(p),\im d_p\pi,H)$ and as a vertical morphism to $(q,L)\mapsto {(q,L,T_{\G}|_{q})}$. These maps are constrained and co-constrained to form a second diagram 
\[
\begin{tikzcd} & I_2 \arrow[d] \\ I_1 \arrow[r] & I_3 
\end{tikzcd} 
\]
where the new schemes involved are 
\begin{align*} 
I_1 &= \Big\{(p,H)\in C^0_k(\pi)\times (\P^n)^\vee \, :\, \im d_p\pi \subseteq H\Big\},\\ I_2 &= \Big\{(q,L)\in(\P^n\setminus Z)\times G \, :\, q\in L \subseteq T_{\G}|_{q}\Big\},\\ I_3 &= \Big\{(q,L,H)\in \P^n\times G\times (\P^n)^\vee\, :\, q\in L \subseteq H\Big\}.
\end{align*}
Note that the pullback of this last diagram is equal to the scheme-theoretical intersection $\Tang(\pi,\G)\cap C_k^0(\pi)$. By Kleiman's Transversality Theorem, there exists an open $U\subseteq\Aut(\P^n)$ such that $\Tang(\pi,\sigma^\ast\G)\cap C_k^0(\pi)$ is generically reduced and of dimension
\[
\dim \Tang(\pi,\sigma^\ast\G)\cap C_k^0(\pi) = \dim I_1 + \dim I_2 - \dim I_3
\]
for all $\sigma\in U$. Thinking of $I_1$, $I_2$ and $I_3$ as fibrations one can check that
\begin{align*}
\dim I_1 &= \dim C^0_k(\pi) + \dim G(1,n-k) = \dim C^0_k(\pi) + (n-k-1),\\
\dim I_2 &= n + \dim G(k,n-1) = n+ k(n-k-1),\\
\dim I_3 &= n + \dim G(k+1,n) + k = n + (k+1)(n-k-1) + k.
\end{align*}
Then we deduce that
\begin{align*}
\dim \Tang(\pi,\G)\cap C_k^0(\pi) & = \dim I_1 + \dim I_2 - \dim I_3 \\
&=\dim C_k^0(\pi) - k \leqslant \dim C_k(\pi) - k \leqslant 0, \\
\end{align*}
where the last inequality follows from Proposition \ref{expecteddimension}. This proves that $\Tang(\pi,\G)$ consists of at most finite points. Moreover, around one such point $p$ the pullback foliation $\F_\sigma$ admits a holomorphic local first integral $f\in\C\{\{z_1,\dots,z_m\}\}$. Since $\F$ is defined by the $1$-form $df$ there, $\Tang(\pi,\G)$ is now determined by the equations $\partial_{z_1} f=\cdots=\partial_{z_m} f=0$. Since the tangential scheme is reduced in $p$ we can say that $\mu(f) = \dim \C\{\{z_1,\dots,z_n\}\}/(\partial_{z_1} f,\dots,\partial_{z_n} f) = 1$, and hence $p$ is a Morse singularity.
\end{proof}

\begin{remark}
The number of isolated singularities of these type of pullback foliations $\F = \pi^\ast\G$ were computed in \cite[Chapter 5]{perrella2024}.
\end{remark}

% Bibliografía
\bibliographystyle{alpha}
\bibliography{bibliografia.bib}

\newcommand{\etalchar}[1]{$^{#1}$}
\begin{thebibliography}{CACGLN04}

\bibitem[ACGH85]{arbarello1985geometry}
Enrico Arbarello, Maurizio Cornalba, Phillip Griffiths, and Joe Harris.
\newblock {\em Geometry of Algebraic Curves. Vol. I}, volume 267 of {\em Grundlehren Math. Wiss.}
\newblock Springer-Verlag, New York, 1985.

\bibitem[Art68]{artin1968solutions}
Michael Artin.
\newblock On the solutions of analytic equations.
\newblock {\em Inventiones mathematicae}, 5(4):277--291, 1968.

\bibitem[AVGZ12]{arnold_singularities_vol1}
Vladimir~Igorevich Arnold, Aleksandr~Nikolaevich Varchenko, and Sabir~Medzhidovich Gusein-Zade.
\newblock {\em Singularities of differentiable maps, Volume I: Classification of Critical Points, Caustics and Wave Fronts}.
\newblock Birkhäuser, 2012.

\bibitem[CAC07]{calvo2006note}
Omegar Calvo-Andrade and Fernando Cukierman.
\newblock A note on the j-invariant and foliations.
\newblock {\em Revista Matemática Iberoamericana}, 2007.

\bibitem[CACGLN04]{calvo2004irreducible}
Omegar Calvo-Andrade, Dominique Cerveau, Luis Giraldo, and Alcides Lins~Neto.
\newblock Irreducible components of the space of foliations associated to the affine lie algebra.
\newblock {\em Ergodic Theory and Dynamical Systems}, 24(4):987--1014, 2004.

\bibitem[CLN82]{camacho1982topology}
C{\'e}sar Camacho and Alcides Lins~Neto.
\newblock The topology of integrable differential forms near a singularity.
\newblock {\em Publications Math{\'e}matiques de l'IH{\'E}S}, 55:5--35, 1982.

\bibitem[CLN96]{cerveau1996irreducible}
Dominique Cerveau and Alcides Lins~Neto.
\newblock Irreducible components of the space of holomorphic foliations of degree two in {$\C\P (n), n\geqslant 3$}.
\newblock {\em Annals of mathematics}, 143:577--612, 1996.

\bibitem[CLNE01]{pullbackCLE}
Dominique Cerveau, Alcides Lins~Neto, and Sebastiaan~Johan Edixhoven.
\newblock Pull-back components of the space of holomorphic foliations on {$\P(n), n\geqslant 3$}.
\newblock {\em Journal of Algebraic Geometry}, 10(4):695--711, 2001.

\bibitem[CLNL{\etalchar{+}}06]{cerveau2006algebraic}
Dominique Cerveau, Alcides Lins-Neto, Frank Loray, Jorge~Vit{\'o}rio Pereira, and Fr{\'e}d{\'e}ric Touzet.
\newblock Algebraic reduction theorem for complex codimension one singular foliations.
\newblock {\em Commentarii Mathematici Helvetici}, 81:157--170, 2006.

\bibitem[CM82]{mattei1982d}
Dominique Cerveau and Jean-François Mattei.
\newblock Formes intégrales holomorphes singulières.
\newblock {\em Ast{\'e}risque}, 97, 1982.

\bibitem[CP08]{CukiermanSplit}
Fernando Cukierman and Jorge~Vit{\'o}rio Pereira.
\newblock Stability of holomorphic foliations with split tangent sheaf.
\newblock {\em American journal of mathematics}, 130(2):413--439, 2008.

\bibitem[CPV09]{cukiermanRacionales}
Fernando Cukierman, Jorge~Vit{\'o}rio Pereira, and Israel Vainsencher.
\newblock Stability of foliations induced by rational maps.
\newblock {\em Annales de la Facult{\'e} des sciences de Toulouse: Math{\'e}matiques}, 18(4):685--715, 2009.

\bibitem[CSV06]{cukierman2006singularities}
Fernando Cukierman, Marcio~G. Soares, and Israel Vainsencher.
\newblock Singularities of logarithmic foliations.
\newblock {\em Compositio mathematica}, 142(1):131--142, 2006.

\bibitem[dCLP22]{degree3}
Raphael~Constant da~Costa, Ruben Lizarbe, and Jorge~Vit{\'o}rio Pereira.
\newblock Codimension one foliations of degree three on projective spaces.
\newblock {\em Bulletin des Sciences Math{\'e}matiques}, 174:103092, 2022.

\bibitem[DLP85]{drezet1985fibres}
Jean-Marc Dr{\'e}zet and Joseph Le~Potier.
\newblock Fibr{\'e}s stables et fibr{\'e}s exceptionnels sur {$\P^2$}.
\newblock {\em Annales scientifiques de l'{\'E}cole Normale Sup{\'e}rieure}, 18(2):193--243, 1985.

\bibitem[FL06]{fulton2006connectivity}
William Fulton and Robert Lazarsfeld.
\newblock Connectivity and its applications in algebraic geometry.
\newblock {\em Algebraic Geometry: Proceedings of the Midwest Algebraic Geometry Conference, University of Illinois at Chicago Circle, May 2--3, 1980}, pages 26--92, 2006.

\bibitem[GAMQV22]{pullbackConFede}
Javier Gargiulo~Acea, Ariel Molinuevo, Federico Quallbrunn, and Sebasti{\'a}n Velazquez.
\newblock Stability of pullbacks of foliations on weighted projective spaces.
\newblock {\em arXiv preprint arXiv:2212.12974}, 2022.

\bibitem[GHS83]{girbau1983deformations}
Joan Girbau, Andr{\'e} Haefliger, and Duraiswamy Sundararaman.
\newblock {\em On deformations of transversely holomorphic foliations.}
\newblock Walter de Gruyter, 1983.

\bibitem[GM88]{gomez1988transverse}
Xavier G{\'o}mez-Mont.
\newblock The transverse dynamics of a holomorphic flow.
\newblock {\em Annals of mathematics}, 127(1):49--92, 1988.

\bibitem[GMLN91]{gomez1991structural}
Xavier G{\'o}mez-Mont and Alcides Lins-Neto.
\newblock Structural stability of singular holomorphic foliations having a meromorphic first integral.
\newblock {\em Topology}, 30(3):315--334, 1991.

\bibitem[Gro60]{grothendieck1960techniques}
Alexander Grothendieck.
\newblock Techniques de construction et th{\'e}oremes d’existence en g{\'e}om{\'e}trie alg{\'e}brique. {IV}. {L}es sch{\'e}mas de {H}ilbert.
\newblock {\em S{\'e}minaire Bourbaki}, 6(221):249--276, 1960.

\bibitem[Gro67]{EGAIV.4}
Alexander Grothendieck.
\newblock {\'E}l\'ements de g\'eom\'etrie alg\'ebrique {IV:} {{\'E}tude} locale des sch\'emas et des morphismes de sch\'emas.
\newblock {\em Publications Math\'ematiques de l'IH\'ES}, 1967.

\bibitem[Jou06]{jouanolou2006equations}
Jean-Pierre Jouanolou.
\newblock {\em Equations de Pfaff alg{\'e}briques}, volume 708.
\newblock Springer, 2006.

\bibitem[LN88]{neto1988algebraic}
Alcides Lins~Neto.
\newblock Algebraic solutions of polynomial differential equations and foliations in dimension two.
\newblock {\em Lecture Notes in Math. 1345}, pages 192--231, 1988.

\bibitem[MMQ15]{massri2015kupka}
C{\'e}sar Massri, Ariel Molinuevo, and Federico Quallbrunn.
\newblock The {K}upka scheme and unfoldings.
\newblock {\em Asian Journal of Mathematics}, 22(6):1025--1046, 2015.

\bibitem[Mol16]{molinuevo2016unfoldings}
Ariel Molinuevo.
\newblock Unfoldings and deformations of rational and logarithmic foliations.
\newblock {\em Annales de l'Institut Fourier}, 66(4):1583--1613, 2016.

\bibitem[MQ17]{molinuevo2017camacho}
Ariel Molinuevo and Federico Quallbrunn.
\newblock On the {Camacho-Lins Neto} regularity.
\newblock {\em arXiv preprint arXiv:1706.07508}, 2017.

\bibitem[MQ22]{molinuevo2022singular}
Ariel Molinuevo and Federico Quallbrunn.
\newblock Singular locus of q-logarithmic foliations.
\newblock {\em arXiv preprint arXiv:2208.09089}, 2022.

\bibitem[Per24]{perrella2024}
Pablo Perrella.
\newblock {\em Deformaciones de foliaciones de tipo pullback}.
\newblock PhD thesis, Universidad de Buenos Aires, 2024.

\bibitem[Qua15]{Quallbrunn_2015}
Federico Quallbrunn.
\newblock Families of distributions and {P}faff systems under duality.
\newblock {\em Journal of Singularities}, 11:164--189, 2015.

\bibitem[Sch73]{schlessinger1973rigid}
Michael Schlessinger.
\newblock On rigid singularities.
\newblock {\em Rice Institute Pamphlet-Rice University Studies}, 59(1), 1973.

\bibitem[Ser07]{sernesi2007deformations}
Edoardo Sernesi.
\newblock {\em Deformations of algebraic schemes}, volume 334.
\newblock Springer Science \& Business Media, 2007.

\bibitem[Suw85]{suwa1985determinacy}
Tatsuo Suwa.
\newblock Determinacy of analytic foliation germs.
\newblock {\em Foliations}, 5:427--461, 1985.

\bibitem[Suw92]{suwa1992unfoldings}
Tatsuo Suwa.
\newblock Unfoldings of codimension one complex analytic foliation singularities.
\newblock {\em Hokkaido University Preprint Series in Mathematics}, 173:1--49, 1992.

\bibitem[Vak17]{vakil2017rising}
Ravi Vakil.
\newblock The rising sea: Foundations of algebraic geometry.
\newblock {\em preprint}, 2017.

\bibitem[Vel22]{tesisSeba}
Sebasti{\'a}n Velazquez.
\newblock {\em Espacios de m{\'o}duli de foliaciones, {\'a}lgebras de Lie y variedades t{\'o}ricas}.
\newblock PhD thesis, Universidad de Buenos Aires, 2022.

\bibitem[War83]{warner1983foundations}
Frank~W. Warner.
\newblock {\em Foundations of differentiable manifolds and Lie groups}, volume~94.
\newblock Springer Science \& Business Media, 1983.

\end{thebibliography}

\end{document}